\documentclass[reqno]{amsart}

\usepackage[applemac]{inputenc}
\usepackage[pdftex]{graphicx}
\usepackage[T1]{fontenc} 
\usepackage{amsfonts} 
\usepackage{calc}
\usepackage{dsfont}
\usepackage{bm}
\usepackage{amsthm}
\usepackage[all]{xy}
\usepackage{verbatim}
\usepackage{fancyhdr}
\usepackage{trfsigns}
\usepackage{slashed}
\usepackage{t4phonet}
\usepackage{geometry}
\usepackage{graphicx}
\usepackage{mathrsfs}
\usepackage{pinlabel}
\usepackage{color}
\usepackage{mathtools}
\usepackage{wrapfig}
\usepackage{youngtab}
\usepackage{young}
\usepackage{tikz}
\usepackage{cancel}
\usepackage{enumerate}
\usetikzlibrary{matrix,arrows}
\usepackage{stmaryrd}
\usepackage{mathabx}

\DeclareMathAlphabet{\mathpzc}{OT1}{pzc}{m}{it}

\newtheorem{teo}{Theorem}

\newtheorem{cor}{Corollary}

\newtheorem{ttes}{Example}

\newcommand{\beq}{\begin{equation}}
\newcommand{\eeq}{\end{equation}}
\newcommand{\bear}{\begin{eqnarray}}
\newcommand{\eear}{\end{eqnarray}}
\newcommand{\longhookrightarrow}{\ensuremath{\lhook\joinrel\relbar\joinrel\rightarrow}}

\newcommand{\slanttwo}[2]{{\raisebox{.2em}{$#1$} \big/ \raisebox{-.2em}{$#2$}}}

 \newcommand{\virgolette}{``}

\newcommand{\set}[1]{\ensuremath{\mathbb{#1}}}

\newcommand{\slantone}[2]{{\raisebox{.1em}{$#1$}\left/\raisebox{-.1em}{$#2$}\right.}}
\newcommand*{\defeq}{\mathrel{\vcenter{\baselineskip0.5ex \lineskiplimit0pt
                     \hbox{\scriptsize.}\hbox{\scriptsize.}}}%
                     =}

\newcommand\asim{\mathrel{%
  \ooalign{\raise0.1ex\hbox{$\sim$}\cr\hidewidth\raise-0.8ex\hbox{\scalebox{0.9}{$\scriptstyle{x}$}}\hidewidth\cr}}}

\newcommand{\proj}[1]{\ensuremath{\mathbb{P}^{#1}}}

\newcommand{\acts}{\mathrel{\reflectbox{$\righttoleftarrow$}}}
\newcommand{\mani}{\ensuremath{\mathpzc{M}}}
\newcommand{\manir}{\ensuremath{\mathpzc{M}_{red}}}
\newcommand{\stsheafm}{\ensuremath{\mathcal{O}_{\mathpzc{M}}}}
\newcommand{\stsheafred}{\ensuremath{\mathcal{O}_{\mathpzc{M}_{red}}}}
\newcommand{\gsheaf}[2]{\ensuremath{\mathcal{#1}_{\mathpzc{#2}}}}
\newcommand{\smani}[1]{\ensuremath{\mathpzc{#1}}}

\begin{document}

\title{PROJECTIVE SUPERSPACES IN PRACTICE}


\author{SERGIO LUIGI CACCIATORI}
\address{Dipartimento di Scienza e Alta Tecnologia, Università dell'Insubria, Via Valleggio 11, 22100 Como, Italy}
\email{sergio.cacciatori@uninsubria.it}

\author{SIMONE NOJA}
\address{Dipartimento di Matematica, Università degli Studi di Milano, Via Saldini 50, 20133 Milano, Italy}
\email{simone.noja@unimi.it}



\begin{abstract} 

This paper is devoted to the study of supergeometry of complex projective superspaces $\mathbb{P}^{n|m}$. First, we provide formulas for the cohomology of invertible sheaves of the form $\mathcal{O}_{\mathbb{P}^{n|m}} (\ell)$, that are pull-back of ordinary invertible sheaves on the reduced variety $\mathbb{P}^n$. Next, by studying the even Picard group $\mbox{Pic}_0 (\mathbb{P}^{n|m})$, classifying invertible sheaves of rank $1|0$, we show that the sheaves $\mathcal{O}_{\mathbb {P}^{n|m}} (\ell)$ are \emph{not} the only invertible sheaves on $\proj{n|m}$, but there are also new genuinely supersymmetric invertible sheaves that are unipotent elements in the even Picard group. We study the $\Pi$-Picard group $\mbox{Pic}_\Pi (\mathbb{P}^{n|m})$, classifying $\Pi$-invertible sheaves of rank $1|1$, proving that there are also non-split $\Pi$-invertible sheaves on supercurves $\proj {1|m}$. Further, we investigate infinitesimal automorphisms and first order deformations of $\proj {n|m}$, by studying the cohomology of the tangent sheaf using a supersymmetric generalisation of the Euler exact sequence. A special special attention is paid to the meaningful case of supercurves $\mathbb{P}^{1|m}$ and of Calabi-Yau's $\mathbb{P}^{n|n+1}$. Last, with an eye to applications to physics, we show in full detail how to endow $\mathbb{P}^{1|2}$ with the structure of $\mathcal{N}=2$ super Riemann surface and we obtain its SUSY-preserving infinitesimal automorphisms from first principles, that prove to be the Lie superalgebra $\mathfrak{osp} (2|2)$. A particular effort has been devoted to keep the exposition as concrete and explicit as possible.

\end{abstract}

\maketitle

\tableofcontents

\section{Introduction}

\noindent The aim of this paper is to study the supergeometry of complex projective superspaces in some depth. Indeed, even if projective superspaces are considered well-understood supermanifolds - they are \emph{split} supermanifolds and different realisations are known (\cite{Cat1} \cite{Cat2} \cite{Manin} \cite{ManinNC}) - and they entered several well-known formal constructions in theoretical physics (see for example \cite{AgaVafa} \cite{BeiJia} \cite{Witten}), some of their geometric structures, constructions and properties have never been established on a solid basis and investigated in detail. In this paper we would like to fill this gap and give a complete a rigorous treatment of the subject.

After having reviewed the main definitions in the theory of supermanifolds and the various constructions of projective superspaces, we concentrate on invertible sheaves of rank $1|0$ on $\proj {n|m}$. We carry out a detailed computation of \v{C}ech cohomology of the sheaves of the form $\mathcal{O}_{\proj {n|m}} (\ell)$, thus providing a (partial) supersymmetric analog of the celebrated \emph{Bott formulas} for ordinary projective spaces.

Then, we study the \emph{even Picard group} of projective superspaces, $\mbox{Pic}_0\, (\proj {n|m})$, that classifies locally-free sheaves of rank $1|0$ over $\proj {n|m}$. In particular, we show that in the case of the supercurves $\proj {1|m}$ the even Picard group has a continuous part and we give the explicit form of its generators, proving that there exist genuinely supersymmetric invertible sheaves on $\proj {n|m}$ that do not come from any ordinary invertible sheaves $\mathcal{O}_{\proj n} (\ell)$ on $\proj n$. These prove to be non-trivial geometric objects, indeed they have in general a non-trivial cohomology, as we show by mean of an example. 

Further, we show that the case of the supercurves $\proj {1|m}$ proves to be special also when looking at the \emph{$\Pi$-invertible sheaves}, that is sheaves of rank $1|1$ with a certain \virgolette exchange symmetry'' between their even and odd part, called $\Pi$-symmetry. Indeed, when looking at $\Pi$-invertible sheaves on $\proj {1|m}$, one finds that for $m>2$, beside the usual \emph{split} $\Pi$-invertible sheaves of the kind $\mathcal{L} \oplus \Pi \mathcal{L}$, for $\mathcal{L}$ a locally-free sheaf of rank $1|0$ on $\proj {1|m}$, one gets \emph{non-split} $\Pi$-invertible sheaves as well, that cannot be presented in the split form.

Later, we study the cotangent sheaf on a generic, possibly non-projected, supermanifold, establishing the short exact sequences it fits into. We then specialise to the case of projected supermanifolds, proving that in this case the Berezinian sheaf - that plays a fundamental role in the theory of integration on supermanifolds - can be reconstructed by means of elements defined of the reduced variety, i.e. the canonical bundle of the reduced variety $\manir$ itself and the determinant of the fermionic sheaf, that is actually a sheaf of locally-free $\mathcal{O}_{\manir}$-modules. This allows to establish on rigorous basis a result that has been already used, especially in theoretical physics: $\mbox{Ber} (\proj {n|m}) \cong \mathcal{O}_{\proj {n|m}} (m-n-1).$ In particular, we call supermanifolds having trivial Berezinian sheaf - such as $\proj {n|n+1}$ for each $n\geq1$ - \emph{generalised Calabi-Yau supermanifolds} (henceforth Calabi-Yau supermanifolds, see \cite{CNR} and \cite{1DCY} for some issues related to this definition in supergeometry), by similarity with the ordinary setting. This is so because the Berezinian sheaf somehow plays the role of canonical sheaf on ordinary manifolds: indeed, as in the ordinary setting one integrates sections of canonical sheaf, in a supersymmetric setting one integrates sections of the Berezinian sheaf instead.     
 
Next, using a supersymmetric generalisation of Euler exact sequence, we study the cohomology of the tangent sheaf of $\proj {n|m}$, which is related to the \emph{infinitesimal automorphisms} and the \emph{first order deformations} of $\proj {n|m}$. In this context, we find that supercurves over $\proj 1$ yield again the richest scenario, allowing for many deformations as their odd dimension increases. 

Then, after we have dealt with the case of supercurves, the example of the Calabi-Yau supermanifold $\proj {1|2}$ is examined. In particular, we show in full details how to endow $\proj {1|2}$ with a structure of $\mathcal{N} = 2$ \emph{super Riemann surface}. In this context, we show how to recover from first principles the $\mathcal{N} =2$ SUSY-preserving automorphisms of $\proj {1|2}$ when structured as a $\mathcal{N}=2$ super Riemann surface. These SUSY-preserving automorphisms prove to be isomorphic to the Lie superalgebra $\mathfrak{osp}(2|2)$, we give a physically relevant presentation of, by exhibiting a particularly meaningful system of generators and displaying their structure equations.

Next we briefly address how to give the structure of $\mathcal{N} =2 $ semi-rigid Riemann surface to $\proj {1|2}$ by means of a certain \emph{twist}: this supermanifold is the basic ingredient in genus 0 \emph{topological string theory}, a topic that we will addressed from a purely supergeometric point of view in a forthcoming paper. To conclude, we study the supergeometry of the so-called twistorial Calabi-Yau $\proj {3|4}$, that enters many constructions in theoretical physics. A supersymmetric generalisation of the exponential exact sequence - we call it \emph{even} exponential exact sequence -, that is used throughout the paper is proved in Appendix A.\\

\noindent {\bf Acknowledgments:} the authors would like to thank Riccardo Re for several illuminating discussions and useful comments on the draft of this paper.

\section{General Definitions and Notations}

\noindent Before we start, we give some fundamental definitions in supergeometry and we fix some notations For a more thorough treatment of the general theory of supermanifolds, see for example \cite{Manin}.

First of all, on the algebraic side, we recall that a superalgebra is a $\mathbb{Z}_2$-graded algebra $A=A_0 \oplus A_1$, whose \emph{even} elements commute and whose \emph{odd} elements anti-commute. \\
The most basic example of complex supermanifold is the complex superspace $\mathbb{C}^{n|m}\defeq ( \mathbb{C}^{n} , \mathcal{O}_{\mathbb{C}^{n}} \otimes \bigwedge [\xi_1, \ldots, \xi_m])$, that is the ringed space whose underlying topological space is given by $\mathbb{C}^{n}$ endowed with the complex topology and whose structure sheaf is given by the sheaf of superalgebras obtained by tensoring the sheaf of holomorphic function on $\mathbb{C}^{n}$ with the Grassmann algebra generated by $m$ elements $\{ \xi_1, \ldots, \xi_m\}.$

More in general, in the same fashion of the ordinary theory of complex manifold, a \emph{complex supermanifold} of dimension $n|m$ is a locally ringed space $(|\mathpzc{M}|, \mathcal{O}_{\mathpzc{M}})$, for $|\mani|$ a topological space and $\stsheafm$ a sheaf of superalgebras, that it is \emph{locally} isomorphic to $\mathbb{C}^{n|m}$ (in the $\mathbb{Z}_2$-graded sense).  
Clearly, a morphism of supermanifolds $\varphi : \mani \rightarrow \mathpzc{N}$ is therefore a morphism of locally ringed space, that is a pair
\begin{align}
(\phi, \phi^\sharp ) : (|\mani| , \stsheafm) \longrightarrow (|\mathpzc{N}|, \mathcal{O}_{\mathpzc{N}} )
\end{align}
where $\phi : |\mani| \rightarrow |\mathpzc{N}|$ is a continuous map and $\phi^\sharp : \mathcal{O}_{\mathpzc{N}} \rightarrow \phi_* \stsheafm$ is a morphism of sheaves of superalgebras (thus respecting the $\mathbb{Z}_2$-grading). These are the kind of morphisms defining the local isomorphisms entering the definition of supermanifold given above.\\
To every supermanifold $\mani $ is attached its reduced manifold $\manir$, consisting, as a locally ringed space, of a pair $(|\mani|, \stsheafred):$ this is an ordinary manifold, whose structure sheaf $\stsheafred$ is given by the quotient $\slantone{\stsheafm}{\mathcal{J}_{\mani}}$, where $\mathcal{J}_\mani$ is the \emph{nilpotent sheaf}, actually a sheaf of ideals in $\stsheafm$ (looking at it as a usual classical scheme, $\manir$, is a reduced scheme, while $\mani$ is not). In other words, given any supermanifold $\mani$, this corresponds to the existence of a \emph{close embedding} as follows:  
\bear
\iota \defeq ( i , i^\sharp ) :  \left ( | \mani | , \slantone{\stsheafm}{\mathcal{J}_\mani} \right ) \rightarrow (| \mani |, \stsheafm).
\eear
Collectively, we will just denote it by $\iota: \manir \rightarrow \mani $. By the way, what really matter here, is that a supermanifold comes endowed with a sheaf morphism $i^\sharp : \stsheafm \rightarrow \stsheafred $ (notice that $i : |\mani | \rightarrow |\mani|$ is the identity, for the underlying topological space remains the same!). Being $\iota : \manir \rightarrow \mani $ a closed embdedding, we have that $i^\sharp : \stsheafm \rightarrow \stsheafred$ is surjective on the stalks, therefore this leads to the existence of the following short exact sequence of sheaves of $\stsheafm$-modules over $|\mani|$: 
\bear
\xymatrix@R=1.5pt{ 
0 \ar[rr] && \mathcal{J}_\mani \ar[rr] & &  \stsheafm \ar[rr] && \slantone{\stsheafm}{\mathcal{J}_\mani} \ar[rr] && 0,
 }
\eear
which we will call the \emph{structural exact sequence} for $\mani$. \\
The above short exact sequence does \emph{not} split in general and if it does the supermanifold $\mani$ is said to be \emph{projected}. Indeed, a splitting corresponds to the existence of a morphism of supermanifolds $\pi : \mani \rightarrow \manir$ given by
\bear
\pi \defeq (\,p , p^\sharp ) : (| \mani |, \stsheafm ) \longrightarrow (|\mani| , \stsheafred),
\eear
splitting the short exact sequence above: 
\bear
\xymatrix@R=1.5pt{ 
0 \ar[rr] && \mathcal{J}_\mani \ar[rr] & &  \stsheafm  \ar[rr]_{\iota} && \ar@{-->}@/_1.3pc/[ll]_{\pi} \slantone{\stsheafm}{\mathcal{J}_\mani} \ar[rr] && 0
 }
\eear
In case such a splitting ($\pi \circ \iota = id_{\manir}$) exists, one has that \emph{globally}, $\stsheafm = \stsheafred \oplus \mathcal{J}_{\mani}$ and the structure sheaf $\stsheafm$ becomes a sheaf of $\stsheafred$-module also, because: 
\bear
\xymatrix@R=1.5pt{ 
\left ( \stsheafred \otimes_{\gsheaf O M} \stsheafm \right ) (U) \ar[r] & \stsheafm (U) \nonumber \\
f \otimes s \ar@{|->}[r] &  \pi^\sharp_U (f) \cdot s.
}
\eear
Notice that in general the projection above does not exist: in this case the supermanifold is said \emph{non-projected} and its structure sheaf is \emph{not} a sheaf of $\stsheafred$-modules, see \cite{CNR} for more details and examples of non-projected supermanifolds over projective spaces. 

Given a supermanifold $\mani$ such that $\dim_{\mathbb{C}} \mani = n|m$, and its nilpotent sheaf $\mathcal{J}_\mani$, we can form a $\mathcal{J}_\mani$-\emph{adic filtration on} $\stsheafm$ of length $m$, that is
\bear
\mathcal{J}_\mani^0 \defeq \stsheafm \supset \mathcal{J}_\mani \supset \mathcal{J}_\mani^2 \supset \mathcal{J}_\mani^3 \supset \ldots \supset \mathcal{J}_\mani^{m} \supset \mathcal{J}_\mani^{m+1} = 0.
\eear
We call $\mbox{{Gr}}^{(i)} \stsheafm \defeq \slantone{\mathcal{J}_\mani^i}{\mathcal{J}_\mani^{i+1}}$ the $\mathcal{J}_\mani^i$-adic component of $\stsheafm$ and we define the following $\set{Z}_2$-graded sheaf
\bear
\mbox{{Gr}}\, \stsheafm \defeq \bigoplus_{i=0}^m \mbox{{Gr}}^{(i)}  \stsheafm = \stsheafred \oplus \slantone{\mathcal{J}_\mani}{\mathcal{J}^2_{\mani}} \oplus \ldots \oplus \slantone{\mathcal{J}_\mani^{m-1}}{\mathcal{J}^m_{\mani}} \oplus {\mathcal{J}^m_\mani}.
\eear
where the $\set{Z}_2$-grading is obtained by taking the obvious $\set{Z}$-grading $\mbox{\emph{mod}}\,2$. Then, we call the superspace $\mbox{{Gr}}\, \mani \defeq (|\mani|, \mbox{{Gr}}\, \stsheafm)$ the \emph{split supermanifold associated to} $\mani$. By definition, every supermanifold $\mani$ is \emph{locally} isomorphic to $\mbox{{Gr}}\, \mani$. We say that a supermanifold $\mani$ is \emph{split} if it is \emph{globally} isomorphic to $\mbox{{Gr}}\, \mani$.\\
Moreover, because of its importance, we reserve the name of \emph{fermionic sheaf} to the first component of the filtration above, $\mbox{{Gr}}^{(1)} \defeq \mathcal{J}_\mani / \mathcal{J}^2_\mani$ and we denote it with $\mathcal{F}_\mani.$ Notice this is a locally-free sheaf of $\stsheafred$-modules of rank $0|m$.\vspace{0.1cm}\\



\noindent For future use, we also briefly recall the notion of \emph{direct image sheaf} and \emph{inverse image sheaf}. \\
Given a continuous map $\phi : |\mani| \rightarrow |\mathpzc{N}|$ between two topological spaces and a sheaf $\mathcal{F}$ over $ | \mani |$, we say that the direct image sheaf of $\mathcal{F}$ is defined as
\begin{align}
|\mathpzc{N}| \owns V \longmapsto \phi_* \mathcal{F} (V) \defeq \mathcal{F} ( \phi^{-1} (V)). 
\end{align}
This is nothing but a (sort of) \emph{push-forward} operation, for it makes a sheaf over $|\mathpzc{N}|$ out of a sheaf over $|\mani|$. \\
The inverse image sheaf cannot be defined the same way, indeed given a sheaf $\mathcal{G}$ over $|\mathpzc{N}|$, it cannot be defined as $|\mathpzc{M}| \owns U \mapsto \phi^* G (U) = \mathcal{G} (\phi (U))$ for in general $\phi (U)$ is not an open set in $|\mathpzc{N}|$, therefore instead of $\phi (U)$ we take its \virgolette best approximation'' as open set and we define the inverse image sheaf as follows
\begin{align}
|\mani | \owns U \longmapsto \phi^{-1} \mathcal{G} (U) \defeq \widetilde{ \varinjlim_{ V \supseteq \phi (U)} \mathcal{G} (V)}.
\end{align}
Note that the above assignation defines a pre-sheaf, so we need to take its sheafification in order to have a sheaf over $|\mani|.$ This, instead, corresponds to a (sort of) \emph{pull-back} operation, indeed we get a sheaf over $|\mani|$ out of a one over $|\mathpzc{N}|.$ \\
Actually, the situation changes a bit when dealing with locally ringed space or schemes and sheaves over them. We therefore consider two locally ringed spaces or schemes $\mani \defeq (|\mathpzc{M}|, \mathcal{O}_{\mathpzc{M}})$ and $\mathpzc{N} \defeq (|\mathpzc{N}|, \mathcal{O}_{\mathpzc{N}} )$ together with a morphism $\varphi: \mani \rightarrow \mathpzc{N}$ between them. Given a sheaf $\mathcal{F}$ of $\mathcal{O}_\mani$-modules, we can take its \emph{push-forward} $\varphi_* \mathcal{F}$ as above: this defines a sheaf of $(\varphi_* \mathcal{O}_\mani)$-modules, but thanks to $\phi^\sharp : \mathcal{O}_{\mathpzc{N}} \rightarrow \phi_* \mathcal{O}_{\mani}$ this becomes naturally a sheaf of $\mathcal{O}_{\mathpzc{N}}$-modules. Indeed, taking $V \in |\mathpzc{N}|$ we have 
\bear
\xymatrix@R=1.5pt{ 
\phi^\sharp_V :  \gsheaf O N \ar[r]  & \gsheaf {O} {M}  (\phi^{-1} (V)) \\
s \ar@{|->}[r] &  \phi^\sharp_V (s)
}
\eear
and we can give a $\gsheaf {O} N$-module structure on $\varphi_* \mathcal{F}$ as follows
\begin{align}
(\gsheaf O N \otimes_{\phi_{\scriptscriptstyle \ast} \gsheaf O M }\varphi_* \mathcal{F}) ( V) & \longrightarrow \varphi_* \mathcal{F} (V) \\
s \otimes t & \longmapsto f^\sharp_V (s) \cdot t. 
\end{align}
Given a sheaf of $\gsheaf O N$-modules, instead, we can define its \emph{pull-back} sheaf as the sheaf of $\gsheaf O M$-modules defined as follows
\bear
U \longmapsto \varphi^{*} \mathcal{G} (U)\defeq (\gsheaf O M \otimes_{\phi^{-1} \gsheaf O N} \phi^{-1} \mathcal{G}) (U). 
\eear
Notice that here it is mandatory to take the tensor product with $\gsheaf O M$ in order to have that $\varphi^*\mathcal{G}$ is indeed a sheaf of $\gsheaf O M$-modules (otherwise it would just be a sheaf of $(\phi^{-1} \gsheaf O N)$-modules).

\section{Projective Superspaces and Cohomology of $\mathcal{O}_{\mathbb{P}^{n|m}} (\ell)$}

Given the usual complex superspace $
\mathbb{C}^{n+1|m} \defeq  ( \mathbb{C}^{n+1} , \mathcal{O}_{\mathbb{C}^{n+1}} \otimes \bigwedge [\xi_1, \ldots, \xi_m])$, one can form the superspace $(\mathbb{C}^{n+1|m})^\times$ simply by considering the obvious restriction of $\mathbb{C}^{n+1|m}$ to the open set $\mathbb{C}^\times \defeq \mathbb{C}^{n+1}\, \setminus \,\{0\}$.

As in \cite{Manin}, the (complex) projective superspace $\proj {n|m}$ is the supermanifold obtained as the quotient of $(\mathbb{C}^{n+1|m})^\times$ by the action $\mathbb{C}^\ast \acts (\mathbb{C}^{n+1|m})^\times$ defined as
\begin{align}
\lambda \cdot (x_0,\ldots,x_n,\xi_1,\ldots,\xi_m) \defeq (\lambda x_0,\ldots,\lambda x_n,\lambda\xi_1,\ldots,\lambda \xi_m)
\end{align}
where $\lambda $ is an element of the multiplicative group $\mathbb{C}^{\ast} \defeq \mathbb{C} \, \setminus \, \{ 0\}$.

Following again \cite{Manin}, we have that $\proj{n|m} $ is canonically isomorphic to $\mbox{Gr}\, \proj{n|m}$, thus it is a \emph{split} supermanifold. As we have seen, it follows that there exists a \emph{projection} $\pi : \proj {n |m} \rightarrow \proj {n}$, more precisely 
$
(id_{|\proj{n}|}, \pi^{\sharp}) : (\proj {n} , \mathcal{O}_{\proj {n|m}}) \rightarrow (\proj {n }, \mathcal{O}_{\proj n})
$
where $\pi^\sharp : \mathcal{O}_{\proj n} \rightarrow id_*\mathcal{O}_{\proj {n|m}}$ is a homomorphism of sheaves that embeds $\mathcal{O}_{\proj {n}}$ into $\mathcal{O}_{\proj {n|m}}$ and it endows $\mathcal{O}_{\proj{n|m}}$ with the structure of sheaf of $\mathcal{O}_{\proj n}$-modules. Given the projection $\pi : \proj {n|m} \rightarrow \proj {n}$, for the sake of precision, we should denote by $\pi_* \mathcal{O}_{\proj{n|m}}$ the structure sheaf of $\proj {n|m}$ when looked at as a sheaf of $\mathcal{O}_{\proj {n}}$-module. Since this specification does not make any difference for our purposes, we will simply keep writing $\mathcal{O}_{\proj {n|m}}$ instead of $\pi_* \mathcal{O}_{\proj{n|m}}$.\\
In particular, we have the splitting as $\mathcal{O}_{\proj n}$-modules 
\bear \label{fact}
\pi_*\mathcal{O}_{\proj {n|m}} = \bigoplus_{k = 0}^{\lfloor m/2 \rfloor} \mathcal{O}_{\proj n} (-2k )^{\oplus {m \choose 2k}} \oplus \Pi \bigoplus_{k = 0}^{\lfloor m/2 \rfloor - \delta_{0, {m {\mbox{\tiny{mod}}}2}}} \mathcal{O}_{\proj n} (-2k-1)^{\oplus { m \choose 2k+1}},
\eear
where we have inserted the symbol $\Pi$ as a reminding for the parity reversing.\\
The $\ell$-shifted sheaf $\pi_* \mathcal{O}_{\proj {n|m}} (\ell)$ therefore is given by
\bear
\pi_* \mathcal{O}_{\proj {n|m}} (\ell )= \bigoplus_{k = 0}^{\lfloor m/2 \rfloor} \mathcal{O}_{\proj n} (-2k + \ell )^{\oplus {m \choose {2k}}} \oplus \Pi \bigoplus_{k = 0}^{\lfloor m/2 \rfloor - \delta_{0, {m {\mbox{\tiny{mod}}}2}}} \mathcal{O}_{\proj n} (-2k-1 + \ell)^{\oplus { {m \choose {2k+1}}}}.
\eear
Thus, we can use the well-known results about the cohomology of $\mathcal{O}_{\proj n} (\ell)$ to compute the cohomology of $\pi_* \mathcal{O}_{\proj {n|m}} (\ell)$: by the way, here and in what follows, for the sake of brevity we will denote it simply as $\mathcal{O}_{\proj {n|m}} (\ell)$ since, again, due to the projected/split property it does not make any difference for our ends. We recall that 
\bear
h^0 ( \mathcal{O}_{\proj n} (\ell )) = {\ell+n \choose \ell} \qquad \quad h^n ( \mathcal{O}_{\proj n} (\ell )) = {|\ell| -1 \choose |\ell | - n -1 }
\eear
where $\ell \geq 0 $ in the first equality and $\ell < 0 $ and $|\ell| \geq n+1$ in the second equality.\\
It is an easy consequence of the previous decomposition that when dealing with $\proj {n|\star}$ we will only have $0$-th and $n$-th cohomology, while all the other cases will be zero. \\
Before we start, we also observe that we will consider together the \emph{even} and \emph{odd} dimensions of the cohomology, by looking at the sheaf of $\mathcal{O}_{\proj n}$-modules above simply as
\bear
\mathcal{O}_{\proj {n|m}} (\ell )= \bigoplus_{k = 0}^{m} \mathcal{O}_{\proj n} (-k + \ell)^{\oplus {m \choose {k}}}, 
\eear  
this will make the combinatorics easier. It will be nice and useful to separate even and odd dimension at some point. \\
We start considering the $0$-th cohomology of $\mathcal{O}_{\proj {n|m}} (\ell)$. We have to deal with two cases: when $m < \ell$, and therefore all the bits in the decomposition are contributing, and when $m \geq \ell$ and therefore just the first $\ell$ contribute.
\begin{itemize}
\item ${m < \ell}:$ in this case we sum over all the contributions:
\begin{align}
h^0 (\mathcal{O}_{\proj{n|m}} (\ell)) &=  \sum_{k = 0}^{m} {m \choose k}{\ell - k + n \choose \ell -k } = \sum_{k = 0}^m \frac{m! (\ell -k + n)!}{(m-k)! k! (\ell - k )! n!} \nonumber \\
& =  \frac{1}{n!} \sum_{k = 0}^m {m \choose k} \frac{(\ell - k + n)!}{(\ell-k)!} = \frac{1}{n!} \sum_{k = 0}^m {m \choose k} (\ell - k + n)\cdot \ldots \cdot (\ell-k +1)  \nonumber \\
& = \frac{1}{n!} \sum_{k = 0}^m {m \choose k}  \left [ \frac{d^n}{dx^n} x^{\ell - k + n} \right ]_{x = 1} = \frac{1}{n!} \frac{d^n}{dx^n} \left [ x^{\ell +n } \left (1 + \frac{1}{x}\right )^{m}  \right  ]_{x=1}
 \nonumber \\
& = \frac{1}{n!} \frac{d^n}{dx^n} \left [ (x+1)^{\ell +n  - m } (x+2)^m \right ]_{x = 0 } 
\end{align}
\item $m \geq \ell:$ in this case we only sum over the first $\ell$ contributions:
\begin{align}
h^0 (\mathcal{O}_{\proj{n|m}} (\ell)) & = \sum_{k = 0}^{\ell} {m \choose k}{\ell - k + n \choose \ell -k } = \sum_{k = 0}^\ell \frac{m! (\ell -k + n)!}{(m-k)! k! (\ell - k )! n!} \nonumber \\
& =  \frac{m!}{n! \, \ell !} \sum_{k = 0}^\ell {\ell \choose k} \frac{(\ell - k + n)!}{(m-k)!} = \frac{m!}{n! \, \ell !} \sum_{k = 0}^\ell {\ell \choose k} (\ell - k + n) \cdot \ldots \cdot ( m - k +1)  \nonumber \\
& = \frac{m!}{n! \, \ell !} \sum_{k = 0}^\ell {\ell \choose k}  \left [ \frac{d^{\ell + n- m} }{dx^{\ell + n- m}} x^{\ell -k  + n } \right ]_{x = 1} =  \frac{m!}{n! \, \ell !} \frac{d^{\ell + n - m}}{dx^{\ell +n -m}} \left [ x^{\ell +n} \left (1 + \frac{1}{x}\right )^{\ell}  \right  ]_{x=1}
 \nonumber \\
& = \frac{m!}{n! \, \ell !} \frac{d^{\ell + n -m}}{dx^{\ell + n - m}} \left [ (x+1)^{n } (x+2)^\ell \right ]_{x = 0 } 
\end{align}
\end{itemize}
We now keep our attention on the contribution given by the $n$-th cohomology. Again, one needs to distinguish between two cases: namely, when $\ell + n + 1 \leq 0 $ we have that all the bits in the decomposition contribute to the cohomology, while if $\ell + n + 1 > 0$ we find that the only bits contributing are the ones having $k \geq \ell + n + 1$.  
\begin{itemize}
\item $\ell + n + 1 \leq 0:$ in this case we sum over all the contributions:
\begin{align}
h^n (\mathcal{O}_{\proj {n|m}} (\ell)) & = \sum_{k = 0}^m {m \choose k } {{k - \ell -1} \choose {k- \ell - n - 1}} = \sum_{k = 0}^m \frac{m! (k- \ell -1)!}{(m-k)! k! n! (k- \ell - n -1)!} \nonumber \\ 
& = \frac{1}{n!} \sum_{k = 0}^m {m \choose k } \frac{(k- \ell -1)!}{(k- \ell - n - 1)!} = \frac{1}{n!} \sum_{k = 0}^m {m \choose k } (k- \ell - 1) \cdot \ldots \cdot (k - \ell - n ) \nonumber \\
& = \frac{1}{n!} \sum_{k = 0}^m {m \choose k } \left [ \frac{d^n}{dx^n} x^{k+ |\ell | -1}  \right ]_{x=1} = \frac{1}{n!} \left [ \frac{d^n}{dx^n} x^{ |\ell | -1} (1+x)^{m}  \right ]_{x=1}  \nonumber \\
& = \frac{1}{n!} \left [ \frac{d^n}{dx^n} (x +1)^{ |\ell | -1} (x+2)^{m}  \right ]_{x=0}
\end{align} 
Actually, this holds in the case $1 \leq |\ell | \leq n$. In the special sub-case $\ell = -1$, one finds:
\begin{align}
h^n (\mathcal{O}_{\proj {n|m}} (-1)) & = \frac{1}{n!} \left [ \frac{d^n}{dx^n} (x+2)^{m}\right ] = \frac{1}{n!}m \cdot (m-1)\cdot  \ldots \cdot (m-n + 1) \cdot 2^{m-n} \nonumber \\
& = {m \choose n} \cdot 2^{m-n} 
\end{align} 
\item $\ell + n + 1 > 0 :$ in this case, the first contribution comes at $k = \ell + n + 1$:
\begin{align}
h^n (\mathcal{O}_{\proj {n|m}} (\ell)) & = \sum_{k = \ell + n + 1}^m {m \choose k } {{k - \ell -1} \choose {k- \ell - n - 1}} =  \frac{1}{n!} \sum_{k = \ell  + 1}^m {m \choose k } \left [ \frac{d^n}{dx^n} x^{k- \ell  -1}  \right ]_{x=1} \nonumber \\ 
= & \frac{1}{n!} \left [ \frac{d^n}{dx^n}  \frac{1}{x^{\ell + 1}} \left ( (x+1)^m - \sum_{k = 0 }^{\ell} {\ell \choose k } x^k \right )  \right ]_{x= 1}\nonumber \\
= &  \frac{1}{n!} \left [ \frac{d^n}{dx^n}  \frac{1}{(x+1)^{\ell + 1}} \left ( (x+2)^m - \sum_{k = 0 }^{\ell} {\ell \choose k } (x+1)^k \right )  \right ]_{x= 0}
\end{align} 
where we stress that we have changed the sum from $\ell + n + 1$ to $\ell + 1$ for the derivative kills the relative terms, which therefore do not give contribution. Also, this holds for $k \geq \ell + n + 1 \leq m.$
\end{itemize}
For the sake of notation we introduce the following definitions: 
\bear
\chi_{m<\ell} (n | m ; \ell) \defeq \frac{1}{n!} \frac{d^n}{dx^n} \left [ (x+1)^{\ell +n  - m } (x+2)^m \right ]_{x = 0 } 
\eear
\bear
\chi_{m \geq \ell } (n | m ; \ell ) \defeq \frac{m!}{n! \, \ell !} \frac{d^{\ell + n -m}}{dx^{\ell + n - m}} \left [ (x+1)^{n } (x+2)^\ell \right ]_{x = 0 }
\eear
\bear
\zeta_{\ell + n + 1 \leq 0} ( n |m ; \ell) \defeq \frac{1}{n!} \left [ \frac{d^n}{dx^n} (x +1)^{ |\ell | -1} (x+2)^{m}  \right ]_{x=0}
\eear
\bear
\zeta_{\ell + n +1 >0} (n |m ; \ell )\defeq \frac{1}{n!} \left [ \frac{d^n}{dx^n}  \frac{1}{(x+1)^{\ell + 1}} \left ( (x+2)^m - \sum_{k = 0 }^{\ell} {\ell \choose k } (x+1)^k \right )  \right ]_{x= 0}
\eear
In conclusion, we have thus proved the following theorem.
\begin{teo} \label{dimensions} Let $\mathcal{O}_{\proj {n |m} } (\ell)$ be the sheaf of $\mathcal{O}_{\proj {n|m}}$-modules as above. Then one has the following dimensions in cohomology:  
\bear
h^i (\mathcal{O}_{\proj {n|m}} (\ell)) = \left \{ \begin{array}{llll} 
\chi_{m < \ell } ( n |m ; \ell)& & \quad & i = 0, \, m < \ell \\  
\chi_{m \geq \ell } ( n |m ; \ell) & & \quad & i = 0, \, m \geq \ell \\ 
\zeta_{\ell + n + 1 \leq 0} ( n |m ; \ell) & & \quad & i = n,\, \ell + n + 1 \leq 0 \\
\zeta_{\ell + n +1 >0} (n |m ; \ell )& &  \quad & i = n, \, \ell + n + 1 > 0 \\
\end{array} 
\right. 
\eear
All the other cohomologies are null.
\end{teo}

\section{Even Picard Group and $\Pi$-Picard Group}

\noindent The \emph{even} invertible sheaves, that is the locally-free sheaves of $\mathcal{O}_\mani$-modules of rank $1|0$ are classified by the so-called \emph{even} Picard group $\mbox{Pic}_0 (\mani)$, that can be proved to be such that $\mbox{Pic}_0 (\mani) \cong H^1 (\mathcal{O}^\ast_{\mani,0})$ (where $\mathcal{O}^\ast_{\mani, 0}$ is a sheaf of \emph{abelian} group), as one might easily get by similarity with the ordinary case. Actually, the above theorem gives the cohomology of \emph{all} possible even invertible sheaves on $\proj {n|m}$, for $n \geq 2$ - in other words all of the even invertible sheaves are of the form $\mathcal{O}_{\proj {n|m}} (\ell)$ for some $\ell$ in the case $n\geq 2$. On the contrary, the situation is very different for supercurves of the kind $\proj {1|m}$, as we prove in the following
\begin{teo}[Even Picard Group for $\proj {n|m}$] The even Picard group of the projective superspace $\proj {n|m}$ is given by 
\bear
\mbox{\emph{Pic}}_0 (\proj {n|m}) \cong \left \{ \begin{array}{lrrl}
\mathbb{Z}\oplus \mathbb{C}^{\{2^{m-2} (m-2) +1\}} & & & n=1, \; m\geq 2 \\
\mathbb{Z} & & & \mbox{else}
\end{array}
\right.
\eear
\end{teo}  
\begin{proof} The main tool to be used in order to compute the even Picard group is the \emph{even} exponential short exact sequence, that is
\bear
\xymatrix@R=1.5pt{ 
0 \ar[rr] && \mathbb{Z}_\mani \ar[rr] &&  \mathcal{O}_{\mani,0} \ar[rr] && \mathcal{O}^\ast_{\mani,0} \ar[rr] && 0.
}
\eear
where $\mathcal{O}_{\mani,0}$ is the even part of structure sheaf of $\mani,$ and likewise for $\mathcal{O}_{\mani, 0}^\ast$. We now consider separately the case $n\geq 3$, $n=2$ and $n=1$.
\begin{itemize}
\item[$n\geq 2$] This is the easiest case, as one has $H^i(\mathcal{O}_{\proj {n|m},0}) =0$ for $i =1,2.$ So the part of the long exact cohomology sequence we are interested into reduces to 
\bear
\xymatrix@R=1.5pt{ 
0 \ar[rr] && \mbox{Pic}_0 (\proj {n|m}) \ar[rr] &&  H^2 (\mathbb{Z}_{\proj n}) \cong \mathbb{Z} \ar[rr] && 0, \nonumber
}
\eear
so that one has $\mbox{Pic}_0 (\proj {n|m}) \cong \mathbb{Z}.$
\item[$n = 2$] The long exact cohomology sequence reduces to 
\bear
\xymatrix@R=1.5pt{ 
0 \ar[r] & \mbox{Pic}_0 (\proj {2|m}) \ar[r] &  H^2 (\mathbb{Z}_{\proj 2}) \cong \mathbb{Z}  \ar[r] & H^2 (\mathcal{O}_{\proj {2|m},0}) \ar[r] & H^2 (\mathcal{O}^\ast_{\proj {2|m},0}) \ar[r] & 0, \nonumber
}
\eear
this splits to give $\mbox{Pic}_0 (\proj {2|m}) \cong \mathbb{Z}$ and $H^2 (\mathcal{O}_{\proj {2|m},0}) \cong H^2 (\mathcal{O}^\ast_{\proj {2|m},0})$.
\item[$n=1$] This is the richest case, as one finds
\bear
\xymatrix@R=1.5pt{ 
0 \ar[r] &   H^1 (\mathcal{O}_{\proj {1|m},0}) \ar[r] & \mbox{Pic}_0 (\proj {1|m}) \ar[r] & H^2 (\mathbb{Z}_{\proj 1})\cong \mathbb{Z} \ar[r] & 0, \nonumber
}
\eear
computing the dimension of $H^1 (\mathcal{O}_{\proj {1|m},0})$, one has 
\bear
h^1 (\mathcal{O}_{\proj {1|m},0}) = \sum_{k= 1}^{\lfloor m/2 \rfloor} {m \choose 2k}  (2k-1) = 2^{m-2}(m-2) +1
\eear
Indeed, one can observe that
\bear \nonumber
\sum_{k= 1}^{\lfloor m/2 \rfloor} {m \choose 2k}  (2k-1) = - \sum_{k = 1}^{\lfloor m/2 \rfloor} {m \choose 2k} + \sum_{k = 1}^{\lfloor m/2 \rfloor} 2k {m \choose 2k} 
\eear
and using that 
\bear \nonumber
(1 + \epsilon x)^m = \sum_{k = 0}^m \epsilon^k x^k {m \choose j} \quad \rightsquigarrow \quad x \frac{d}{dx} \left ( 1 + \epsilon x \right )^m = \sum^{m}_{k = 0} k \epsilon^k x^k {m \choose k}
\eear 
so that for $m \geq 2$ one might write
\bear \nonumber
x \frac{d}{dx}\left [ (1 + x )^m + (1-x)^m \right ] = 2 \cdot \sum_{k = 1}^{\lfloor m/2 \rfloor} 2k x^k {m \choose 2k }. 
\eear
At $x = 1$, the sums yields
\bear \nonumber
m 2^{m-2} = \sum_{k = 1}^{\lfloor m/2 \rfloor} 2k {m \choose 2k }
\eear \nonumber 
Putting the two bits together one finds
\begin{align}
- \sum_{k = 1}^{\lfloor m/2 \rfloor} {m \choose 2k} + \sum_{k = 1}^{\lfloor m/2 \rfloor} 2k {m \choose 2k} = (- 2^{m-1} +1) + (m 2^{m-2}) = (m-2)2^{m-2} +1.
\end{align}
So that the conclusion follows, 
\bear
\mbox{Pic}_0 (\proj {1|m}) \cong \mathbb{Z}\oplus \mathbb{C}^{\{2^{m-2} (m-2) +1\}} 
\eear
which concludes the proof.
\end{itemize}
\end{proof}
\noindent The previous theorem tells us that all of the invertible sheaves on $\proj {n|m}$ for $n \geq 2$ are of the form $\mathcal{O}_{\proj {n|m} } (\ell)$. In other words, we can say that all of the invertible sheaves on $\proj {n|m}$ for $n \geq 2$ are the pull-backs via the projection $\pi : \proj {n|m} \rightarrow \proj n$ of the invertible sheaves $\mathcal{O}_{\proj n} (\ell)$ on $\proj n.$ This is no longer true in the one-dimensional case: indeed over $\proj {1|m}$, for $m \geq 2$, there are invertible sheaves that cannot the obtained by the pull-back of a certain invertible sheaf $\mathcal{O}_{\proj 1} (\ell)$ via $\pi : \proj {1|m} \rightarrow \proj {1}$, that is there are \emph{genuinely supersymmetric invertible sheaves on $\proj {1|m}$}: in view of this even supergeometry of projected, actually split supermanifolds, could effectively become a richer geometric setting compared to its ordinary counterpart.\\
In the following theorem we provide the explicit form of the transition functions of the invertible sheaves on the supercurves $\proj {1|m}$.
\begin{teo}[Invertible Sheaves on $\proj {1|m}$] The invertible sheaves on $\proj {1|m}$ are generated by the following transition functions
\bear
\mbox{\emph{Pic}}_0 (\proj {1|m}) \cong \left \langle  w^k , 1 + \sum_{|[\underline s] | = 1}^{\lfloor m/2 \rfloor } \sum_{\ell  = 1}^{2 |[\underline s]| - 1} c_{\ell}^{[\underline s]} \frac{\psi^{[\underline s]}}{w^\ell}\right \rangle. 
\eear
where $k \in \mathbb{Z}$ and $c_\ell^{[\underline s]} \in \mathbb{C}$ for each $|[\underline s]| = 1, \ldots, \lfloor m/2 \rfloor$ and $\ell = 2|[\underline s]| -1$.
\end{teo} 
\begin{proof} One has to explicitly compute the representative of $\mbox{Pic}_0 (\proj {1|m}) \cong H^1(\mathcal{O}^\ast_{\proj {1|m},0})$. In order to achieve the usual covering of $\proj {1}$ given by the two open sets $\{\mathcal{U}, \mathcal{V}\}$ can be used, so that one has 
\begin{align}
&C^0 (\{\mathcal{U}, \mathcal{V} \}, \mathcal{O}_{\proj {1|m},0}^\ast) = \mathcal{O}^\ast_{\proj {1|m}, 0} (\mathcal{U}) \times \mathcal{O}^\ast_{\proj {1|m},0} (\mathcal{V}) \\
&C^1 (\{\mathcal{U}, \mathcal{V} \}, \mathcal{O}_{\proj {1|m},0}^\ast) = \mathcal{O}^\ast_{\proj {1|m}, 0} (\mathcal{U} \cap \mathcal{V}).
\end{align}
The \v{C}ech $0$-cochains are thus given by pairs of elements of the type $(P (z, \theta_1, \ldots, \theta_m), Q (w, \psi_1, \ldots, \psi_m))$. In order to write the expressions of the elements $(P,Q)$ we use the following notation: we set $[\underline s] = \{ i_1, \ldots, i_m\} $ to be a multi-index with $ i_k = \{ 0, 1\}$ such that $|[\underline s] | = \sum_{k=1}^m i_k \leq m$ and we put 
\bear
\theta^{[\underline s]} \defeq \theta_1^{i_1} \ldots \theta_k^{i_k} \ldots \theta^{i_m}_m,
\eear  
where, clearly, $\theta^0_k = 1_\mathbb{C}$. We can thus write
\begin{align}
P (z, \theta_1, \ldots, \theta_m) & = a + \sum_{k=1}^{\lfloor m/2 \rfloor}\sum_{|[\underline s]| = 2k} \tilde{P}_{[\underline s]} (z) \theta^{[\underline s]} = \nonumber \\
& = a + \sum_{i < j}^m \tilde{P}_{ij} (z) \theta^i \theta^j + \sum_{i<j<k<l = 1}^m \tilde{P}_{ijkl} (z)\theta^i \theta^j \theta^k \theta^l + \ldots
\end{align}
\begin{align}
Q (w, \psi_1, \ldots, \psi_m) & = b + \sum_{k=1}^{\lfloor m/2 \rfloor}\sum_{|[\underline s]| = 2k} \tilde{Q}_{[\underline s]} (w) \psi^{[\underline s]} = \nonumber \\
& = b + \sum_{i < j}^m \tilde{Q}_{ij} (w) \psi^i \psi^j + \sum_{i<j<k<l = 1}^m \tilde{Q}_{ijkl} (w)\psi^i \psi^j \psi^k \psi^l + \ldots
\end{align}
where $a, b \in \mathbb{C}^\ast.$ The boundary maps $\delta : C^0 (\{\mathcal{U}, \mathcal{V} \}, \mathcal{O}_{\proj {1|m},0}^\ast) \rightarrow C^1 (\{\mathcal{U}, \mathcal{V} \}, \mathcal{O}_{\proj {1|m},0}^\ast)$ acts as 
\bear
\delta ((P, Q)) = Q (w, \psi_1, \ldots, \psi_m) P^{-1} (z, \theta_1, \ldots, \theta_m)\big \lfloor_{\mathcal{U}\cap \mathcal{V}}.
\eear
Explicitly, one finds
\begin{align} \label{cob}
\delta ((P,Q)) & = \frac{b}{a} + \sum_{i < j=1}^m \left ( \frac{\tilde{Q}_{ij}(w)}{a} + \frac{b}{a^2} \frac{\tilde{P}_{ij} (1/w)}{w^2}\right ) \psi_i \psi_j + \nonumber \\
& + \sum_{i <j < k < l =1}^m \left ( \frac{\tilde{Q}_{ijkl} (w)}{a} - \frac{b}{a^3} \frac{\tilde{P}_{ijkl} (1/w)}{w^4} - \frac{1}{a^2} \frac{ \tilde{Q}_{ij}(w) \tilde{P}_{kl} (1/w)}{w^2} \right ) \psi_i \psi_j \psi_k \psi_l + \ldots
\end{align} 
Clearly, one immediately sees that $H^0(\mathcal{O}_{\proj {1|m},0}^\ast) \cong \mathbb{C}^\ast$, as the group is represented by the constant cocycles $ (a, a) $ with $a \neq 0.$\\
On the other hand, the elements in $\mathcal{O}^\ast_{\proj {1|m}, 0} (\mathcal{U}\cap \mathcal{V})$ are given by expressions having the following form
\begin{align} \label{cocy}
W (w, 1/w, \psi_1, \ldots, \psi_m) & = c w^k + \sum_{k=1}^{\lfloor m/2 \rfloor}\sum_{|[\underline s]| = 2k} \tilde{W}_{[\underline s]} (w, 1/w) \psi^{[\underline s]}  = \nonumber \\
& = c w^k + \sum_{i < j}^m \tilde{W}_{ij} (w, 1/w) \psi^i \psi^j + \sum_{i<j<k<l = 1}^m \tilde{W}_{ijkl} (w, 1/w) \psi^i \psi^j \psi^k \psi^l + \ldots
\end{align}
where again, clearly $c \in \mathbb{C}^\ast$, $k\in \mathbb{Z}$ and $\tilde{W}_{[\underline s]} \in \mathbb{C}[w, 1/w]$ for all the multi-index $[\underline s]$. Confronting the expressions in \eqref{cob} and \eqref{cocy} one see that 
\begin{itemize}
\item $b/a$ can be used to set the coefficient $c$ of $w^k$ to $1$;
\item For every power in the $\theta$'s, the polynomials $\tilde{Q}_{[\underline s]} (w)$ kill the regular part the corresponding $\tilde{W}_{[\underline s]}$;
\item The mixed terms, such as for example ${ \tilde{Q}_{ij}(w) \tilde{P}_{kl} (1/w)}/{w^2},$ in \eqref{cob} does not interfere anyway, as they enter in lower-order powers in the theta's, so that they are completely fixed. 
\end{itemize}
We thus see that the non-exact $1$-cocycles are given by transition functions having the following form 
\bear
H^1 (\mathcal{O}^\ast_{\proj {1|m},0}) \cong \left \langle w^k , 1 + \sum_{|[\underline s] | = 1}^{\lfloor m/2 \rfloor } \sum_{\ell  = 1}^{2 |[\underline s]| - 1} c_{\ell}^{[\underline s]} \frac{\psi^{[\underline s]}}{w^\ell} \right \rangle
\eear
where $k \in \mathbb{Z}$ and each of the $(m-2)2^{m-2} +1$ coefficients $c_\ell^{[\underline s]}$ is a complex number.
\end{proof}
 
\noindent Before we go on, we stress that one can check that $\mbox{Pic}_0 (\proj {1|m})$, as seen via the isomorphism with $ \mathbb{Z}\oplus \mathbb{C}^{(m-2)2^{m-2} +1}$, has the structure of an abelian group with addition, that is 
\bear
\xymatrix@R=1.5pt{ 
\mathbb{Z} \oplus (\mathbb{C} \oplus \ldots \oplus \mathbb{C}) \times \mathbb{Z} \oplus (\mathbb{C} \oplus \ldots \oplus \mathbb{C}) \ar[rr] && \mathbb{Z} \oplus (\mathbb{C} \oplus \ldots \oplus \mathbb{C}) \\
( (k, c, \ldots, c_{f(m)}) , (\tilde k, \tilde c, \ldots, \tilde c_{f(m)}) ) \ar@{|->}[rr] && (k+ \tilde k, c + \tilde c, \ldots c_{f(m)} + \tilde c_{f(m)} ).
}
\eear
where $f(m) = (m-2)2^{m-2} +1.$\\
It is fair to say that, if on the one hand we have been able to compute the cohomology of the invertible sheaves of the kind $\mathcal{O}_{\proj {n|m}} (k)$ (actually, pull-back of some $\mathcal{O}_{\proj n} (k)$ by $\pi : \mani \rightarrow \manir$), it is instead not certainly a trivial task to deduce a general formula for the cohomology of the most general invertible supersymmetric sheaf on $\proj {1|n}$ for $n\geq 2$, originating by tensor product of the generators shown above. \\
At this stage, it would be easy to provide a general formula for the genuinely supersymmetric generators of the even Picard group above, but this would not help to solve the general question. We thus limit ourselves to provide the reader with an example, as to show that these invertible sheaves have an interesting non-trivial cohomology. 
\begin{ttes}[The Cohomology of a Supersymmetric Invertible Sheaf] We consider the following supersymmetric invertible sheaf on $\proj {1|3}$:
\bear
\mathcal{L}_{\proj {1|3}} \defeq \left \{ \{ \mathcal{U}, \mathcal{V} \} ,\quad e_{\mathcal{U}} = \left ( 1 + \sum_{i < j; i,j= 1}^3 \frac{\psi_{i} \psi_{j}}{w}\right ) e_{\mathcal{V}} \right \}
\eear
for $e_\mathcal{U}$ and $e_{\mathcal{V}}$ two local frames on the open sets $\mathcal{U}$ and $\mathcal{V}$ respectively. Notice this is a generator of the even Picard group for $\proj {1|3}$. It is easy to actually compute \v{C}ech cohomology. We have that 
\begin{align}
& C^{0} (\{ \mathcal{U}, \mathcal{V} \}, \mathcal{L}_{\proj {1|3}}) \defeq \mathcal{L}_{\proj {1|3}} (\mathcal{U})\times \mathcal{L}_{\proj {1|3}} (\mathcal{V}) \owns \left (  P(z, \theta_1, \theta_2) e_{\mathcal{U}}, \; Q (w, \psi_1, \psi_2) e_{\mathcal{V}} \right ) \nonumber \\
&C^{1} (\{ \mathcal{U}, \mathcal{V} \}, \mathcal{L}_{\proj {1|3}} ) \defeq \mathcal{L}_{\proj {1|3}} (\mathcal{U} \cap \mathcal{V})\owns W (w, 1/w, \psi_1, \psi_2) e_{\mathcal{V}}
\end{align}
where $P \in \mathbb{C}[z, \theta_1, \theta_2], \; Q\in \mathbb{C}[w, \psi_1 \psi_2]$ and $W \in \mathbb{C}[w, 1/w, \psi_1, \psi_2]$. \\
By following the usual strategy, we change coordinates as to get
\begin{align}
P (z, \theta_1, \theta_2) e_{\mathcal{U}} & = \left (A^{(0)} (z) + \sum_{i = 1}^3 A^{(1)}_i (z) \theta_i + \sum_{i<j; i ,j =1}^3 A^{(2)}_{ij}(z) \theta_i \theta_j + A^{(3)}(z) \theta_1 \theta_2 \theta_3 \right ) e_{\mathcal{U}} \nonumber \\ 
& = \Bigg ( A^{(0)} (1/w) + \sum_{i = 1}^3 A^{(1)}_i (1/w) \frac{\psi_i }{w} +  \sum_{i<j; i ,j =1}^3 \left ( \frac{A^{(2)}_{ij} (1/w)}{w^2} + \frac{A^{(0)} (1/w)}{w} \right ) \psi_{i} \psi_j + \nonumber  \\
& \quad + \sum_{i<j; i,j =1}^3 \left ( \sum_{i = 1}^3 (-1)^{i-1}\frac{ A^{(1)}_{i} (1/w)}{w^2} + \frac{A^{(3)} (1/w)}{w^3} \right ) \psi_1 \psi_2 \psi_3 \Bigg ) e_{\mathcal{V}}.
\end{align}
One can clearly see that there is no way to get a globally defined holomorphic section, that is to extend $P(z, \theta_1, \theta_2) e_{\mathcal{U}}$ to the whole $\proj {1|3}$ without hitting a singularity, and this tells that $h^0 (\mathcal{L}_{\proj {1|3}}) = 0|0$. \\
Instead, considering $\left ( Q  -  P\right )\lfloor_{\mathcal{U} \cap \mathcal{V}}$, upon using the expression above for $P$ in the chart $\mathcal{V}$, one finds that $h^1 (\mathcal{L}_{\proj {1|3}}) = 3 |2$, and in particular, it is generated by the following elements:
\begin{align}
H^1 ( \mathcal{L}_{\proj {1|3}} ) = \left  \langle \frac{\psi_1 \psi_2}{w}, \frac{\psi_1 \psi_3}{w}, \frac{\psi_2 \psi_3}{w} \; \Bigg |\;  \frac{\psi_1 \psi_2 \psi_3}{w}, \frac{\psi_1 \psi_2 \psi_3}{w^2} \right \rangle_{\mathbb{C}}  
\end{align}
where we have written the representative in the chart $\mathcal{V}$: notice that all of these elements are nilpotent, they live in $\mathcal{J}_{\proj {1|3}} (\mathcal{U} \cap \mathcal{V})$. The cohomology of $\mathcal{L}_{\proj {1|3}}$ is thus given by
\bear
h^i (\mathcal{L}_{\proj {1|3}}) = \left \{ \begin{array}{l} 
0|0 \qquad \quad i = 0 \\
3|2  \qquad \quad i = 1.
\end{array}
\right.
\eear
Similar computation can be easily done for any invertible sheaves of this kind: in general, one would again a vanishing zeroth cohomology group, while a non-vanishing - and possibly very rich as the fermionic dimension of $\proj {1|m}$ increases - first cohomology group.  
\end{ttes}
Actually, the case of the supercurves differs from the higher-dimensional case also when looking at the $\Pi$-invertible sheaves, we call them $\mathcal{L}_\Pi$ (see in particular \cite{Manin}, \cite{ManinNC} for an introduction to the matter and \cite{Noja} for a recent construction of $\Pi$-projective spaces, the supermanifolds supporting these particular kind of sheaves): these are \virgolette special'' sheaves of rank $1|1$ on $\mani$ endowed with an \emph{odd} endomorphism $\Pi : \mathcal{L}_{\Pi} \rightarrow \mathcal{L}_\Pi$, exchanging their even with their odd part. \\
Given a locally-free sheaf of $\mathcal{O}_\mani$-modules $\mathcal{L}$ on $\mani$, a trivial, or better, \emph{split} $\Pi$-invertible sheaf on $\mani$ is simply given by $\mathcal{L}^{s}_{\Pi} \defeq \mathcal{L} \oplus \Pi \mathcal{L}$ and the odd endomorphism $\Pi$ acts as the exchange of factors. But there might be non-trivial (non-split) $\Pi$-invertible sheaves. Indeed, following Manin, the set of isomorphism classes of $\Pi$-invertible sheaves, call it $\mbox{Pic}_{\Pi} (\mani)$ is isomorphic to the cohomology group $H^1 (\mathcal{O}^\ast_{\mani})$. Notice that this is \emph{not} actually a group, but actually a pointed-set, as $\mathcal{O}^\ast_{\mani}$ is \emph{not} a sheaf of abelian groups. This pointed-set, in turn fits into an exact sequence
\bear 
\xymatrix@R=1.5pt{ 
\cdots \ar[r] & \mbox{Pic}_0 (\mani) \ar[r] & \mbox{Pic}_{\Pi} (\mani) \ar[r] & H^1 (\mathcal{O}_{\mani, 1}) \ar[r]^\delta & H^2(\mathcal{O}^\ast_{\mani,0}). \nonumber
}
\eear
The first map above is defined as follows: 
\bear 
\xymatrix@R=1.5pt{ 
\mbox{Pic}_0(\mani) \ar[rr] && \mbox{Pic}_{\Pi} (\mani)  \\
\mathcal{L} \ar@{|->}[rr] && \mathcal{L} \oplus \Pi \mathcal{L}.
}
\eear
It is clear that, because of exactness, the obstruction to split a $\Pi$-invertible sheaf lies in the kernel of the boundary map $\delta$. We thus have the following simple corollaries to the previous theorem
\begin{cor}[$\Pi$-Invertible Sheaves on $\proj {n|m}$ for $n>1$] All of the $\Pi$-invertible sheaves on $\proj {n|m}$ for $n>1$ are split, that is they are of the form
\bear
\mathcal{L}^s_{\Pi} = \mathcal{O}_{\proj {n|m}} (\ell) \oplus \Pi \mathcal{O}_{\proj {n|m}} (\ell) 
\eear
for some $\mathcal{O}_{\proj {n|m}} (\ell).$
\end{cor}
\begin{proof} In the case $\proj {n|m}$, when $n>1$, one finds $H^1 (\mathcal{O}_{\proj {2|m}, 1}) =0$ for every $m$, therefore one has that the only $\Pi$-invertible sheaves on $\proj {n|m}$ are those of the form $\mathcal{L}^{s}_\Pi =\mathcal{O}_{\proj {2|m}} (\ell) \oplus \Pi \mathcal{O}_{\proj {2|m}}(\ell),$ as we have already proved that $\mathcal{L} = \mathcal{O}_{\proj {n|m}} (\ell)$ are the only invertible sheaves on $\proj {n|m}$ in the case $n>1$. \end{proof}
\begin{cor}[$\Pi$-Invertible Sheaves on $\proj {1|m}$] There are non-split $\Pi$-invertible sheaves on $\proj {1|m}$ if and only if $m > 2$.
\end{cor}
\begin{proof}
On $\proj {1|m}$, one finds that clearly $H^2 (\mathcal{O}^\ast_{\proj {1|m}, 0}) =0$ and $H^0 (\mathcal{O}_{\proj {1|m}, 1}) = 0$ for every $m$. Also, $H^1 (\mathcal{O}_{\proj {1|m}, 1} ) \neq 0 $ for $m\geq 3$, thus there can be non-split $\Pi$-invertible sheaves on $\proj {1|m}$ for $m\geq 3$. In particular 
\bear
h^1 (\mathcal{O}_{\proj {1|m}, 1}) = \sum_{m=1}^{\lfloor m/2 \rfloor - \delta_{0, {m {\mbox{\tiny{mod}}}2}}} { m\choose 2k+1} 2k = 2^{m-2} (m-2),  
\eear 
so that one has a short exact sequence in the case $m\geq 3$, that reads
\bear
\xymatrix@R=1.5pt{ 
0 \ar[r] & \mathbb{Z} \oplus \mathbb{C}^{2^{m-2} (m-2) +1} \ar[r] & \mbox{Pic}_{\Pi} (\mani) \ar[r]  & \mathbb{C}^{2^{m-2} (m-2)} \ar[r]^\delta & 0. \nonumber
}
\eear
One finds that $\ker (\delta ) \cong \mathbb{C}^{2^{m-2} (m-2)}$ (as a pointed-set). \\
The cases $m=1 $ and $m=2$ are special in that one has $H^1 (\mathcal{O}_{\proj {1|m},1}) = 0$, so that the only $\Pi$-invertible sheaves are split of the form $\mathcal{L}_{\Pi}^s = \mathcal{L}\oplus \Pi \mathcal{L}$ for a certain invertible sheaf $\mathcal{L}$ on $\proj {1|m} $, for $m = 1, 2.$
\end{proof}

\section{Cotangent Sheaf and Berezinian Sheaf of a Supermanifold}

\noindent In this section we first work in full generality to see what happens when dealing with the sheaf of $1$-forms on a generic supermanifold $\mani$. Once the general framework is established, we specialise to the case of projected supermanifolds and workout the example of projective superspaces $\proj {n|m}$, we are particularly concerned with.\vspace{.5cm}

It is an early result due to Leites that the tangent sheaf $\mathcal{T}_\mani$ of a (say complex) supermanifold $\mani$ of dimension $p|q$ is locally-free, having a local basis given by the derivations $\{ \partial_{z_1}, \ldots, \partial_{z_p}, \partial_{\theta_1}, \ldots, \partial_{\theta_q} \}$. The cotangent sheaf or sheaf of $1$-forms $\Omega_\mani^1$ is defined as the dual of the tangent sheaf $(\mathcal{T}_\mani)^\vee = \mbox{Hom}_{\stsheafm} (\mathcal{T}_{\mani}, \stsheafm).$ It is locally free as well and a local basis is given by $\{ dz_1, \ldots, dz_p , d\theta_1, \ldots,  d\theta_q \}$, with a duality pairing with the tangent space (locally) given by:  
\bear
\xymatrix@R=1.5pt{ 
\langle \cdot , \cdot \rangle_{U} : \, (\mathcal{T}_\mani \otimes_{\stsheafm} \Omega^1_\mani )(U) \ar[r] & \stsheafm (U) \nonumber \\
D \otimes \omega \ar@{|->}[r] &  \langle D , \omega \rangle_{U} 
}
\eear
if $D $ and $\omega$ are local sections of $\mathcal{T}_\mani$ and $\Omega^1_\mani$ respectively. Given two local sections of the structure sheaf $f, g \in \stsheafm (U)$, the duality paring reads
\bear
\langle f D, g \,\omega \rangle_U = (-1)^{|D| \cdot |g|} f g \, \langle D, \omega \rangle_U.  
\eear

We now consider a generic supermanifold $\mani$, that is, in principle, we only have an embedding $\iota : \manir \rightarrow \mani$, which allows us to have an exact sequence of $\stsheafm$-modules as follows 
\bear
\xymatrix@R=1.5pt{ 
0 \ar[rr] && \mathcal{N}_{\stsheafm} \ar[rr] & &  \Omega^1_{\mani} \ar[rr]^{res_{\stsheafm}} && \iota_* \Omega^1_{\manir} \ar[rr] && 0
 }
\eear
where $\mathcal{N}$ is a suitable sheaf of $\stsheafm$-module, actually kernel of the map $res_{\stsheafm} :  \Omega^1_{\mani}  \rightarrow \iota_* \Omega^1_{\manir}  $, where $\iota_*\Omega^1_{\manir}$ is the push-forward of the sheaf of $1$-forms over the reduced variety $\manir$, that is indeed a sheaf of $\stsheafm$-modules. \\
Likewise, we can also consider the pull-back of the previous short exact sequence: 
\bear
\xymatrix@R=1.5pt{ 
0 \ar[rr] && \mathcal{N}_{\stsheafred} \ar[rr] &&  \iota^* \Omega^1_{\mani} \ar[rr]^{res_{\stsheafred}} && \Omega^1_{\manir} \ar[rr] && 0
}
\eear
This gives a short exact sequence of $\stsheafred$-modules. Here, similarly as above $ \mathcal{N}_{\stsheafred}$ is the kernel. Notice that the pull-back by $\iota $ makes the short exact sequence well-defined for we have $\iota^* \Omega^1_\mani = i^{-1}\Omega^1_\mani \otimes_{i^{-1} \stsheafm} \stsheafred.$\\
We now wonder if there actually exists a projection $\pi : \mani \rightarrow \manir$ splitting the exact sequence above. In presence of the projection, it makes sense to consider the following short exact sequence of $\stsheafm$-modules:
\bear
\xymatrix@R=1.5pt{ 
0 \ar[rr] && \pi^* \Omega^1_{\manir } \ar[rr] &&  \Omega^1_{\mani} \ar[rr] \ar[rr] && \mathcal{Q}_{\stsheafm} \ar[rr] && 0
 }
\eear
where now $\mathcal{Q}_{\stsheafm}$ is a suitable quotient and $\pi^* \Omega^1_{\manir} = \stsheafm \otimes_{p^{-1} \stsheafred} p^{-1} \Omega^1_{\manir}$. This short exact sequence splits, 
\bear
\xymatrix@R=1.5pt{ 
0 \ar[rr] && \pi^* \Omega^1_{\manir }  \ar[rr]_{imm} & &  \ar@{-->}@/_1.3pc/[ll]_{proj}  \Omega^1_{\mani}  \ar[rr] && \mathcal{Q}_{\stsheafm} \ar[rr] && 0.
}
\eear
Notice that $\mathcal{Q}_{\stsheafm}$ is therefore the quotient $\mathcal{Q}_{\stsheafm} \defeq \slanttwo{\Omega^1_{\mani}}{\pi^* \Omega^1_{\manir}} $, so locally, we have that elements in $\mathcal{Q}_{\stsheafm}$ are of the form $\stsheafm \cdot \{ dz_1, \ldots, dz_p, d\theta_1, \ldots, d\theta_q \} \mbox{mod}\, \stsheafm \cdot \{ dz_1, \ldots, dz_p \}.$\\
Locally, over an open set $U \subseteq |\mani |$ we have: 
\bear
\xymatrix@R=1.5pt{ 
\pi^* \Omega^1_{\manir} (U) \ar[r]^{imm_U} & \Omega^1_\mani \ar[r]^{proj_U} (U) & \pi^{*} \Omega^1_{\manir} (U) \nonumber \\
\stsheafm \cdot \{dz_1, \ldots, dz_p \}  \ar@{|->}[r] &  \stsheafm \cdot \{dz_1 , \ldots, dz_p , 0 \ldots, 0 \} \ar@{|->}[r] & \stsheafm \cdot \{ dz_1, \ldots, dz_p  \}. 
 }
\eear
Therefore, when dealing with a projected / split supermanifold that possess a morphism $\pi : \mani \rightarrow \manir $, we can consider the sheaf of $1$-form $\Omega_\mani^1$ as given by a direct sum, as follows: 
\bear
\xymatrix@R=1.5pt{ 
0 \ar[rr] &&\pi^\ast \Omega^1_{\manir} \ar[rr] && \pi^\ast \Omega^1_{\manir} \oplus \mathcal{Q}_{\mathcal{O}_{\mani}} \ar[rr] && \mathcal{Q}_{\mathcal{O}_{\mani}} \ar[rr] && 0.
 }
\eear
Now we need the following
\begin{cor} Let $\mani$ be a projected supermanifold, with projection given by $\pi : \mani \rightarrow \manir$. The the following isomorphism holds
\bear
\pi^\ast \mathcal{F}_\mani \cong \slantone{\Omega^1_\mani}{\pi^\ast \Omega^1_{\manir}}.
\eear
\end{cor}
\begin{proof} 
Locally elements in $\pi^\ast \mathcal{F}_\mani$ can be written as $\theta^a \mbox{mod} \mathcal{J}^2_\mani$ for $a = 1, \ldots, m$ where $m$ is the odd dimension of $\mani,$ while elements in $\slanttwo{\Omega^1_\mani}{\pi^\ast \Omega^1_{\manir}}$ have a local form given by $d\theta^a \, \mbox{mod} \, \pi^\ast \Omega^1_{\manir},$ again for $a= 1, \ldots, m$ where $m$ odd dimension of $\mani$. The isomorphism we are considering reads
\begin{align}
\xymatrix@R=1.5pt{
\pi^\ast \mathcal{F}_\mani \ar[rr] && \slantone{\Omega^1_{\mani}}{\pi^\ast \Omega^1_{\mani}} \\
\theta^j \, \mbox{mod} \, \mathcal{J}_\mani^2 \ar@{|->}[rr] &&  d\theta^j \,\mbox{mod}\, \pi^\ast \Omega^1_{\manir}.
 }
\end{align}
We need this to hold true when passing from chart to chart, that is we need that $ \eta^j \, \mbox{mod}\, \mathcal{J}^2_\mani$ go to $d\eta^j \,\mbox{mod}\, \pi^\ast \Omega^1_{\manir},$ therefore we consider another local chart of $\mani$ having local coordinates given by $(y^i | \theta^j)$, and  we consider the transformation of $dx^i $ and of $d\theta^j$ for 
\begin{align}
d y^i & = \sum_b \frac{\partial y^i}{\partial x^b} dx^b + \sum_b \frac{\partial y^i}{\partial \theta^b} d\theta^b = \sum_b \frac{\partial y^i}{\partial x^b} dx^b \equiv 0\, \mbox{mod}\, \pi^\ast \Omega^1_{\manir},
\end{align}
as $\partial_{\theta^b} y^i = 0 $ since $\mani $ is projected and therefore $y = y (x)$. Moreover, remembering that $\eta^j \equiv \sum_{b} f_b^j(x) \theta^b \, \mbox{mod}\, \mathcal{J}^2_\mani,$ one has
\begin{align}
d \eta^j & = \sum_b \frac{\partial \eta^j}{\partial x^b} dx^b + \sum_b \frac{\partial \eta^j}{\partial \theta^b} d\theta^b \nonumber \\ 
& = \sum_b \frac{\partial}{\partial x^b } \left (  \sum_c f_c^j (x) \theta^c \, \mbox{mod} \, \mathcal{J}^2_\mani \right ) dx^b + \sum_b \frac{\partial}{\partial \theta^b} \left (  \sum_c f_c^j (x) \theta^c \, \mbox{mod} \, \mathcal{J}^2_\mani\right ) d\theta^b \nonumber \\
& = \sum_{b, c } \frac{\partial f^{j}_c (x)}{\partial x^b} \theta^c \, \mbox{mod}\, \mathcal{J}^2_\mani \, dx^b + \sum_{b} f^{j}_b (x) \, \mbox{mod} \, \mathcal{J}^2_\mani \, d\theta^b \nonumber \\
& \equiv \sum_b f^j_b (x) d\theta^b \, \mbox{mod} \, \left ( \pi^\ast \Omega^1_{\manir} \right ), 
\end{align}
thus concluding the proof.
\end{proof}
\noindent Then, the previous short exact sequence can be re-written in the more useful form 
\bear
\xymatrix@R=1.5pt{ 
0 \ar[rr] &&\pi^\ast \Omega^1_{\manir} \ar[rr] && \pi^\ast \Omega^1_{\manir} \oplus \pi^\ast \mathcal{F}_{{\mani}} \ar[rr] && \pi^\ast \mathcal{F}_{{\mani}} \ar[rr] && 0,
 }
\eear
so that one can get the following
\begin{teo}[Berezinian of Projected Supermanifold] \label{bersup} Let $\mani$ be a projected supermanifold, with projection given by $\pi : \mani \rightarrow \manir$, then one has
\bear
\mbox{\emph{Ber}}\, (\Omega^1_\mani) \cong \pi^\ast \left ( \det (\Omega^1_{\manir} )\otimes_{\mathcal{O}_{\manir}} (\det \mathcal{F}_\mani)^{\otimes -1 } \right )
\eear
\end{teo}
\begin{proof} We have seen that in presence of a projection $\pi : \mani \rightarrow \manir$, one has that $\Omega^1_{\mani} \cong \pi^\ast \Omega^1_{\manir} \oplus \pi^\ast \mathcal{F}_{{\mani}} $, then it is enough to take the Berezinian of the both sides of the isomorphism. In particular, the right-hand side reads
\begin{align}
\mbox{Ber}\, (\pi^\ast \Omega^1_{\manir} \oplus \pi^\ast \mathcal{F}_{{\mani}} ) & \cong \mbox{Ber}\, \left ( \pi^\ast \Omega^1_{\manir} \right ) \otimes_{\mathcal{O}_\mani} \mbox{Ber}\, \left ( \pi^\ast \mathcal{F}_\mani \right ) \nonumber \\
& \cong \pi^\ast \left ( \det (\Omega^1_{\manir}) \otimes_{\mathcal{O}_{\manir}} ( \det \mathcal{F}_\mani )^{\otimes -1}  \right ), 
\end{align}
thus completing the proof.
\end{proof}
\noindent This result allows to evaluate the Berezinian of projected supermanifolds by means of completely classical elements: indeed, once there is a projection, what one needs is to know the canonical sheaf $\mathcal{K}_{\manir}\defeq \det (\Omega^1_{\manir})$ of the reduced manifold and the determinant sheaf $\det \mathcal{F}_\mani $ of the fermionic sheaf, that we recall it is a  (locally-free) sheaf of $\mathcal{O}_{\manir}$-modules, that is an object living on the reduced manifold! \\
Using this result, one can easily evaluate the Berezinian sheaf of a projective superspace $\proj {n|m},$ as the following theorem shows. 
\begin{teo}[Berezinian of $\proj {n|m}$ (1)] Let $\proj {n|m}$ be the $n|m$-dimensional projective superspace. Then
\bear
\mbox{\emph{Ber}} (\Omega^1_{\proj {n|m}})\cong \mathcal{O}_{\proj {n|m}} (m-n-1).
\eear
\end{teo}
\begin{proof} In the case of $\proj {n|m}$ it boils down to consider the following split exact sequence
\bear
\xymatrix@R=1.5pt{ 
0 \ar[rr] && \pi^\ast \Omega^1_{\proj n} \ar[rr] && \Omega^1_{\proj {n|m}} \ar[rr] && \pi^\ast (\Pi \mathcal{O}_{\proj n} (-1)^{\oplus m}) \ar[rr] && 0.
 }
\eear 
Therefore, taking the Berezinian of the short exact sequence, one gets
\begin{align}
\mbox{Ber} (\Omega^1_{\proj {n|m}}) & \cong \mbox{Ber}\, \left (\pi^\ast \Omega^1_{\proj n} \oplus \pi^\ast \left ( \Pi \mathcal{O}_{\proj n}^{\oplus m} \right ) \right ) \nonumber \\
& \cong \mbox{Ber}(\pi^\ast \Omega^1_{\proj n}) \otimes_{\mathcal{O}_{\proj {n|m}}} \mbox{Ber} \left (\pi^\ast \left (\Pi \mathcal{O}_{\proj n} (-1)^{\oplus m} \right ) \right) \nonumber \\
& \cong \pi^{\ast} \left ( \mbox{det} (\Omega^1_{\proj n}) \otimes_{\mathcal{O}_{\proj n}} \left (\mbox{det} (\mathcal{O}_{\proj n} (-1)^{\oplus m}) \right )^{\otimes -1}  \right ) \nonumber \\
& \cong \pi^\ast \left ( \mathcal{O}_{\proj n} (-1-n) \otimes_{\mathcal{O}_{\proj n}} \mathcal{O}_{\proj n} (m) \right ) \nonumber \\
& \cong \pi^\ast \left ( \mathcal{O}_{\proj n} (m-n-1)\right ) \nonumber \\
& \cong \mathcal{O}_{\proj {n|m}} (m-n-1),
\end{align}
that yields the conclusion. 
\end{proof}
\subsection{Construction: a \virgolette Super'' First Chern Class} The above result, Theorem \eqref{bersup}, allows to define a supersymmetric analog for the Chern class of an ordinary supermanifold, at least in the case we are dealing with a projected supermanifold. Indeed, if we define 
\bear
c^s_1 (\Omega_\mani^1) \defeq c_1 (\det \Omega^1_{\manir}) - c_1 (\det \mathcal{F}_\mani ),
\eear
then, by imitating the usual definition for an ordinary reduced variety, it makes sense to put 
\bear
c_1^s (\mani) \defeq -c^s_1( \Omega_\mani^1).
\eear
Notice that this reduces to the usual definition of Chern class of a variety in case we set the odd part to zero (recall that $\mathcal{F}_\mani \subseteq \mathcal{J}_\mani$), that is we have $c_1^s (\manir) = c_1 (\manir) = - c_1 (\det \Omega^1_{\manir}).$ \\
This construction immediately gives the following corollary.
\begin{cor}[Super Chern Class of $\proj {n|m}$] Let $\proj {n|m}$ a projective superspace. Then we have
\bear
c_1^s (\proj {n|m}) = n + 1 - m. 
\eear
\end{cor} 
\begin{proof} Since we have that $\mathcal{F}_{\proj {n|m}} = \mathcal{O}_{\proj n} (-1)^{\oplus m} $, we have that 
\begin{align}
c_1^s (\proj {n|m} ) & = - c_1 (\det \Omega^1_{\proj n} ) + c_1 (\mathcal{O}_{\proj n} (-1)^{\oplus m})  \\
& =  - c_1 (\mathcal{O}_{\proj n} (-1-n)) +  m \cdot c_1 (\mathcal{O}_{\proj n} (-1)) \\
& = n + 1 - m, 
\end{align}
that proves the corollary.
\end{proof}
\noindent Actually, as in \cite{CNR} \cite{1DCY}, one can reasonably define a supermanifold $\mani$ to be a Calabi-Yau if it has trivial Berezinian sheaf. We see that, in the case of projected supermanifolds, Theorem \eqref{bersup} reduces the task of verifying the triviality of the Berezianian sheaf to computations on the reduced manifolds, and likewise the super first Chern class defined above give a useful numerical criteria (notice though, that a super-analog of Yau's Theorem does not hold in general, see \cite{1DCY}: so, as usual in supergeometry, some attention is required when addressing questions of this kind).
\begin{ttes}[$\proj {n|n+1}$ is a Calabi-Yau Supermanifold] a well-known fact that can be easily red off the theory developed above is that in the case of a projective superspace $\proj {n|m}$ one satisfies the Calabi-Yau condition choosing $m= n+1$: in other words $\proj {n|n+1}$ for any $n >1$ has trivial Berezinian sheaf and vanishing super first Chern class.
\end{ttes}

\subsection{A glimpse to de Rham Cohomology of $\proj {n|m}$}

For the sake of completeness and in order to give a self-contained and exhaustive treatment of the geometry of projective superspaces, we also report a result concerning the de Rham cohomology of $\proj {n|m}$, referring to \cite{1DCY} for details.
\begin{teo}[de Rham Cohomology of $\proj {n|m}$] The de Rham cohomology of $\proj {n|m}$ is given by
\bear
H_{dR}^{i;j} (\proj{n|m}) = \left \{ \begin{array}{lll} \mathbb R^{m\choose{j}} & &  i= 2k,\ k=0,\ldots,n, \; j=0,\ldots,m, \\
0 & & i=2k+1,\ k=0,\ldots, n-1, \; j=0,\ldots,m.
\end{array}
\right.
\eear
In particular, the generators are given by 
\bear
\omega_{k,I_j}\defeq \wedge^k \omega_{FS} \otimes \bigwedge_{\ell\in I_j}\theta_\ell \delta(d\theta_\ell),
\eear
where $I_j\subseteq \{0,1,\ldots,m\}$ has cardinality $j$, and $\omega_{FS}$ is the ordinary Fubini-Study form.
\end{teo}
\noindent Notice that in the supergeometric setting the de Rham cohomology depends in general on \emph{two} numbers, $i$ and $j$ above, and indeed the de Rham complex in supergeometry is not really a complex, but a bicomplex instead. The first number, denoted as $i$, refers to the actual \emph{degree} of the forms and for $j=0$ it is indeed just the analog of the usual degree of a differential form. The second one, denoted with $j$, is instead a supersymmetric novelty, indeed it refers to the so-called \emph{picture number} associated to a form in supergeometry. Loosely speaking, the picture number of a form tells the number of expressions of the kind $\delta (d\theta_\ell)$ that appear. These, in turn, are needed to make the complex bounded from above in the case $j = m$, something which is crucial for integration on supermanifolds. We will not dwell into these deep and important issues any further though, as they are not directly related to the main goals of the paper. The interested reader might want to look at \cite{Cat2} \cite{1DCY} \cite{IntW} for details and an explanation of the set of problems concerning differential forms and integration on supermanifolds.

\section{Euler Sequence and the Cohomology of Tangent Sheaf of $\proj {n|m}$}

\noindent The main tool that we will exploit here to compute the cohomology of the tangent bundle of the projective super space $\mathbb{P}^{n|m}$ is a generalisation to a supergeometric setting of the ordinary Euler exact sequence, that reads
\bear \label{euler}
\xymatrix{
0 \ar[r] & \mathcal{O}_{\proj {n|m}}  \ar[r] &  \mathcal{O}_{\proj {n|m}} (+1) \otimes \mathbb{C}^{n+1|m}  \ar[r] & \mathcal{T}_{\proj {n|m}} \ar[r]  & 0,
}  
\eear  
we will write $\mathcal{O}_{\proj {n|m}} (1)^{\oplus n+1 |m } = \mathcal{O}_{\proj {n|m}} (+1) \otimes \mathbb{C}^{n+1|m} $. \\
First we notice that this has an easy consequence: 
\begin{teo}[Berezinian of $\proj {n|m}$ (2)] \label{ber2} Let $\proj {n|m}$ be the $n|m$-dimensional projective superspace. Then
\bear
\mbox{\emph{Ber}} (\Omega^1_{\proj {n|m}})\cong \mathcal{O}_{\proj {n|m}} (m-n-1).
\eear
\end{teo}
\begin{proof} We consider the dual of the super version of the Euler exact sequence in \eqref{euler}, that is
\bear
\xymatrix{
0 \ar[r] & \Omega^1_{\proj {n|m}}  \ar[r] &  \mathcal{O}_{\proj {n|m}} (- 1)^{\oplus n+1|m}  \ar[r] & \mathcal{O}_{\proj {n|m}} \ar[r]  & 0.
}  
\eear
Since $\mbox{Ber} (\mathcal{O}_{\proj {n|m}})$ is trivial, taking into account the multiplicative behaviour of the Berezinian with respect to exact sequence, one finds
\begin{align}
\mbox{Ber} (\Omega^1_{\proj {n|m}}) & \cong \mbox{Ber}\, (\mathcal{O}_{\proj {n|m}} (-1)^{\oplus n+1|m}) \cong \mbox{Ber}\, \left (\mathcal{O}_{\proj {n|m}} (-1)^{\oplus n+1} \oplus \Pi \mathcal{O}_{\proj {n|m}} (-1)^{\oplus m} \right ) \nonumber \\ 
& \mathcal{O}_{\proj {n|m}} (-n-1) \otimes_{\mathcal{O}_{\proj {n|m}}} \mathcal{O}_{\proj {n|m}} (m) \cong \mathcal{O}_{\proj {n|m}} (m-n-1),
\end{align}
which concludes the proof.
\end{proof}
\noindent Notice that is nothing but an easier way to get to the same result we obtained above for $\proj {n|m}$. \vspace{.5cm}
 
We now look at the cohomology exact sequence associated to the Euler exact sequence. We find the long exact sequence 
\begin{align} 
\xymatrix@R=.5cm{ 
0 \ar[r] &   H^0 (\mathcal{O}_{\proj {n|m}}) \ar[r]^{\tilde{e}_0} & H^0(\mathcal{O}_{\proj {n|m}} (1)^{\oplus n+1|m} ) \ar[r] & H^0 (\mathcal{T}_{\proj {n|m}} ) \ar[r] & H^1(\mathcal{O}_{\proj {n|m}}) \ar[r] & \ldots \nonumber \\
\ldots \ar[r] & H^{n-1} (\mathcal{T}_{\proj {n|m}} ) \ar[r] & H^n(\mathcal{O}_{\proj {n|m}}) \ar[r]^{\tilde{e}_n} & H^n (\mathcal{O}_{\proj {n|m}} (1)^{\oplus n+1|m} ) \ar[r] & H^n (\mathcal{T}_{\proj {n|m}}) \ar[r] & 0.
}   
\end{align}
These are the only relevant parts of the long exact sequence in cohomology associated to the Euler sequence, since, considering the $\mathcal{O}_{\proj n}$-module structure of the sheaf of algebras $\mathcal{O}_{\proj {n|m}}$ obtained by the projection map $\pi : \proj {n|m} \rightarrow \proj n$, one has the factorisation in a direct sum as in \eqref{fact} and Theorem \ref{dimensions} holds true.\\
Actually, the map $\tilde e_n : H^n(\mathcal{O}_{\proj {n|m}}) \longrightarrow H^n(\mathcal{O}_{\proj {n|m}} (1)\otimes \mathbb{C}^{n+1|m}) $ in cohomology deserves some special attention, indeed the following theorem holds true, as Riccardo Re explained to us.  
\begin{teo} \label{teoCY} The map 
\begin{equation}\label{eq:maxrkmap}
\tilde e_n: H^n(\mathcal{O}_{\proj {n|m}}) \longrightarrow H^n(\mathcal{O}_{\proj {n|m}} (1)\otimes \mathbb{C}^{n+1|m}).
\end{equation}
has maximal rank. In particular it is injective if $m \neq n+1$.
\end{teo}
\begin{proof} We use the Serre duality on a supermanifold (see \cite{VoMaPe}, Proposition 3, for a thorough discussion). The dualising sheaf of $\proj {n |m}$ is given by $\mbox{Ber} (\Omega^1_{\proj {n|m}})$, that is the so called Berezinian sheaf of $\proj {n |m}$: in Theorem \eqref{ber2} we have shown that it is given by $\mathcal{O}_{\proj {n|m}} (m-n-1)$. Given a sheaf $\mathcal{E}_{\proj {n|m}}$ of $\mathcal{O}_{\proj{n|m}}$-module, Serre duality then reads 
\bear
H^i (\mathcal{E}_{\proj {n|m}}) \cong H^{n-i} (\mathcal{E}^\vee_{\proj {n|m}} \otimes \mathcal{O}_{\proj {n|m}} (m-n-1))^{\vee} 
\eear 
By functoriality of Serre duality, we see therefore that the map (\eqref{eq:maxrkmap}) can be written as
\begin{equation}\label{eq:serredualized}
\tilde e_n : H^0(\mathcal{O}_{\proj {n|m}}(m-n-1))^\vee \longrightarrow H^0(\mathcal{O}_{\proj {n|m}}(m-n-2) \otimes \mathbb{C}^{n+1|m})^\vee,
\end{equation} 
which is the dual to the map $ H^0(\mathcal{O}_{\proj {n|m}}(m-n-2)\otimes \mathbb{C}^{n+1|m}) \stackrel{(X_0,\ldots,\Theta_m)}{\longrightarrow} H^0(\mathcal{O}_{\proj {n|m}}(m-n-1))$ defined by multiplication of matrices of global sections. \\
Setting $X^*_i$ and $\Theta_j^*$ to be the \emph{dual} base to $\langle X_0,\ldots, X_n, \Theta_1, \ldots \Theta_m \rangle$, that generates the vector superspace $H^0 (\mathcal{O}_{\proj {n|m}} (1))$, we can consider the superspace $\mathcal{U}_{n+1|m}$, spanned by $\langle X^*_0,\ldots, X^*_n,\Theta^*_1,\ldots,\Theta^*_m \rangle$, and we may write 
\bear
\begin{array}{l}
H^0(\mathcal{O}_{\proj {n|m}}(m-n-1))^\vee=Sym^{m-n-1}\,(\mathcal{U}_{n+1|m})\\ 
H^0(\mathcal{O}_{\proj {n|m}}(m-n-2))^\vee = Sym^{m-n-2}\,(\mathcal{U}_{n+1|m}),
\end{array} 
\eear
where $Sym$ denotes the symmetric power functor in the supercommutative setting. In other words, this actually means that we are writing these spaces as the superspace of the homogeneous forms in $X^*_i,\Theta^*_j$ of global degrees $m-n-1$ and $m-n-2$, respectively.
As usual, the dual operation to the multiplication by a variable $X^*_i$ or $\Theta^*_j$, is the derivation $\partial_{X^*_i}$ or $\partial_{\Theta^*_j}$, respectively. Therefore the map (\eqref{eq:serredualized}) can be written as the {\em super gradient map}
\bear
\xymatrix{ 
\tilde e_n : Sym^{m-n-1} ( \mathcal{U}_{n+1|m} ) \ar[rr]^{\tilde \nabla_{(X^*_i,\Theta^*_j)}} &&  Sym^{m-n-2} (\mathcal{U}_{n+1|m} \otimes \mathbb{C}^{n+1|m}).
} 
\eear
where the super gradient map is given by 
\bear
\tilde \nabla_{(X^*_i,\Theta^*_j)}  
\defeq 
\left ( \begin{array}{c} 
\partial_{X_0^*} \\
\vdots \\
\partial_{X_n^*} \\
-\partial_{\Theta_1^*} \\
\vdots \\
- \partial_{\Theta_m^*}
\end{array}
\right )
\eear
where the minus signs in front of the odd derivatives are due to the super transposition.\\
Now it is obvious by inspection that this map has non-zero kernel if and only if $m=n+1$, in which case the first space consists in the constant homogeneous forms, and the second space is zero.
\end{proof}
\noindent The previous theorem and the knowledge of the cohomology of the sheaves $\mathcal{O}_{\proj {n|m}} (\ell)$ we have achieved early on, allow us to compute the cohomology of the tangent space of projective super spaces $\proj{n |m}$. Notice that, surprisingly, some attention must be paid in the case the projective superspace is Calabi-Yau in the sense explained above (i.e. trivial Berezinian sheaf), corresponding to $m = n+1$. \\
We first consider separately the cohomology of the sheaves that appear in the long exact sequence and then we take on some remarkable examples. 
\begin{itemize}
\item $\mathcal{O}_{\proj {n|m}}:$ first of all it is immediate to see from the previous decomposition that one has $h^0 (\mathcal{O}_{\proj {n|m}}) = 1;$ moreover, by means of the previous theorem, one has 
\bear
h^n (\mathcal{O}_{\proj {n|m}})  = \frac{1}{n!} \left [ \frac{d^n}{dx^n} \frac{1}{x+1} \left ( 1 + (x+2)^m \right )  \right ]_{x= 0} 
\eear
Recall that of course one needs $m\geq n+1$, otherwise we found no cohomology in this case.
\item $\mathcal{O}_{\proj {n|m}} (+1)^{n+1|m}:$ we start observing that, disregarding the parity, we find the following decomposition over $\proj n$
\bear
\mathcal{O}_{\proj {n|m}} (+1) \otimes \mathbb{C}^{ n+1|m} = \bigoplus_{k = 0}^m \mathcal{O}_{\proj {n}} (-k +1)^{\oplus (n+m+1){m \choose k}} \nonumber
\eear
This immediately leads to realise that only the cases $k = 0$ and $k = 1$ contribute to $h^0 (\mathcal{O}_{\proj {n|m}} (+1)^{\oplus n+1|m})$, and indeed one finds
\bear
h^0 (\mathcal{O}_{\proj {n|m}} (+1)^{\oplus n+1|m}) = (n+m +1) (n+1) + (n+m +1) (m) = (n+m +1)^2 \nonumber   
\eear
The $n$-cohomology space can be computed by the above theorem, or directly as follows:
\begin{align}
h^n (\mathcal{O}_{\proj {n|m}} (+1)^{\oplus { n+1|m}}) & = \sum_{k = n+2}^m {m \choose k } { {k-2} \choose {k-2 - n} } = \frac{n+m +1}{n!} \sum_{k = 2}^m {m \choose k} \left [ \frac{d^n}{dx^n} \frac{x^k}{x^2}  \right ]_{x=1} \nonumber \\
& = \frac{n+m +1}{n!} \left [ \frac{d^n}{dx^n} \left \{ \frac{(x+2)^m}{(x+1)^2} - \frac{1}{(x+1)^2} - \frac{m}{(x+1)} \right \}\right ]_{x=0} = \nonumber \\
& = \frac{n+m +1}{n!} \left [ \frac{d^n}{dx^n} \left \{ \frac{(x+2)^m}{(x+1)^2} \right \} - (-1)^n (m+n+1)  \right ]_{x=0}  \nonumber
\end{align} 
\end{itemize}
Putting together these results and counting the dimensions, it is easy to see what happens in the case $n \geq 2$: 
\begin{itemize}
\item[{\bf Automorphisms:}] taking into account the even and odd dimensions, we have that of $h^0 (\mathcal{T}_{\proj {n|m}})$ matches the dimension of $\mathfrak{sl} (n+1|m)$, the Lie superalgebra of the Lie supergroup $PGL(n+1|m)$, as somewhat expected by similarity with the ordinary case on $\proj n$. In particular, we have
\bear
h^0 (\mathcal{T}_{\proj {n|m}}) = n^2 + m^2 + 2n | 2nm + 2m \qquad n\geq 2, \; \forall m,  
\eear
that indeed equals $\dim \mathfrak{sl} (n+1|m).$
\item[{\bf Deformations:}] dimensional reasons assures that, in the case $n>2$, the supermanifold $\proj {n|m}$ is rigid for all $m$. Moreover, in the case $n=2$, Theorem \ref{teoCY} guarantees that when $m \neq 3$, we have $h^1 (\mathcal{T}_{\proj {n|m}}) = 0,$ since $\tilde e_2 : H^2 (\mathcal{O}_{\proj {2|m}}) \rightarrow H^2 (\mathcal{O}_{\proj {2|m}}^{\oplus 3|m})$ is injective and therefore $\proj {2|m}$ is rigid also whenever $m\neq 3$. \\
The only case that actually needs to be treated carefully is that of the Calabi-Yau supermanifold $\proj {2|3}$: indeed, in this case Theorem \ref{teoCY} is not helping us, and further, since we are working over the projective plane $\proj 2$ the second cohomology groups could, in principle, be non-zero. We have, thus, the following exact sequence: 
\bear
0\longrightarrow  H^1 (\mathcal{T}_{\proj {2|3}}) \longrightarrow H^2 (\mathcal{O}_{\proj {2|3}}) \longrightarrow H^2 (\mathcal{O}_{\proj {2|3}} (+1)^{\oplus 3|3} ) \longrightarrow H^2 (\mathcal{T}_{\proj {2|3}}) \longrightarrow 0.
\eear  
A direct computation, or the use of the previous formulas, shows that $H^2 (\mathcal{O}_{\proj {2|3}}) \cong \mathbb{C}^{0|1}$ and $H^2 (\mathcal{O}_{\proj {2|3}} (+1)^{\oplus 3|3} ) = 0$, so one has that $h^1 (\mathcal{T}_{\proj {2|3}}) = 0|1$ and therefore $\proj {2|3}$ possess a single \emph{odd} deformation. This is the only projective superspace having a first order deformation whenever $n \geq 2$. We will see that the situation is much different over $\proj 1.$
\end{itemize}
In the least section of the this paper, even if it fits what we have just said above about projective superspaces having bosonic dimension greater than $2$, we will still treat in detail the case of the Calabi-Yau supermanifold $\proj {3|4}$, that have entered many formal constructions in theoretical physics. Now, we will instead focus our attention on the case of supercurves over $\proj 1.$  

\subsection{Supercurves over \proj {1} and the Calabi-Yau supermanifold \proj{1|2}}
We start considering supercurves of the kind $\proj {1|m}$, where $m \neq 2$. In this case, as seen above, the map $\tilde e_1 : H^1 (\mathcal{O}_{\proj {1|m}}) \rightarrow H^1 (\mathcal{O}_{\proj {1|m}}(1)^{\oplus 2|m})$ is actually injective and the long exact sequence in cohomology splits in two short exact sequences. \\
From the previous theorem, or if preferred, by direct computation, we get the following results for the cohomologies involved,
\bear
\left \{ \begin{array}{l}
h^0 (\mathcal{O}_{\proj {1|m}}) = 1\\
h^1 (\mathcal{O}_{\proj {1|m}}) = (m-2) 2^{m-1} + 1 \\
h^0 (\mathcal{O}_{\proj {1|m}}(+1)^{\oplus 2|m}) = (m+2)^2 \\
h^1 (\mathcal{O}_{\proj {1|m}}(+1)^{\oplus 2|m}) = (m+2) [(m+2) + (m - 4) 2^{m-1}].
\end{array}
\right.
\eear 
This is enough to conclude that 
\bear
h^0 (\mathcal{T}_{\proj {1|m}}) = (m+2)^2 -1 = m^2 + 4m + 3.
\eear
This is what we expect, since this number corresponds to the dimension - actually to the sum of the even and odd dimensions - of the super Lie algebra $\mathfrak{sl} (2|m)$, connected to the super group ${PGL}(2|m)$, the \virgolette superisation'' of the general projective group $PGL (2, \mathbb{C})$, the group of automorphisms of $\proj 1.$  \\
As for the first-order deformations (see also \cite{1DCY}), we finds that 
\bear
h^1(\mathcal{T}_{\proj {1|n}}) =  (m+2) \left [ (m+2) + (m-4)2^{m-1} \right ]- (m-2)2^{m-1}-1, 
\eear
Therefore we can observe that we have no (first-order) deformations in the case of $\proj {1|1}$, $\proj{1|2}$, $\proj {1|3}$ and we start having deformations from $\proj {1|4}$, where we find $h^1(\mathcal{T}_{\proj {1|4}}) = 19$. A more careful analysis, aimed to distinguish between even and odd dimensions, yields: $h^1(\mathcal{T}_{\proj {1|4}}) = 11 | 8 $. \\
Before we go on we notice that, of course, $H^2 (\mathcal{T}_{\proj {1|m}}) = 0,$ therefore following the supersymmetric generalisation of a well-known result by Kodaira and Spencer (\cite{AdvMod}, page 21) due to A. Yu. Vaintrob \cite{Vaintrob}, we have that for any $m \geq 4$, the complex supervariety $\proj {1|m}$ has no \emph{obstruction classes} and there exists a Kuranishi family whose base space is a complex \emph{supermanifold} having indeed dimension equal to $h^1 (\mathcal{T}_{\proj {1|m}}).$ It would be certainly interesting to study this families in details to get acquainted with the - still partially mysterious - \emph{odd deformations} appearing in the theory of supermanifolds.  
\vspace{4pt}\\

We are left with the Calabi-Yau supermanifold $\proj {1|2}$: in this case, the map $\tilde e_1 : H^1 (\mathcal{O}_{\proj {1|2}}) \rightarrow H^1 (\mathcal{O}_{\proj {1|2}}(1)^{\oplus 2|2})$ is not injective, and the long exact sequence does not split into two short exact sequence as for $\proj {1|m}$, $m \neq 2$ and something interesting happens. \\
The key is to observe that in the case $m = 2$ we get $h^1 (\mathcal{O}_{\proj {1|2}} (+1)^{\oplus 2|2}) = 0$, so we immediately have that $h^1 (\mathcal{T}_{\proj {1|2}})= 0$, which tells us that $\proj {1|2}$ is rigid. We are left with the following sequence to evaluate: 
\begin{align} 
0 \longrightarrow H^0 (\mathcal{O}_{\proj {1|2}}) {\longrightarrow} H^0(\mathcal{O}_{\proj {1|2}} (+1)^{\oplus 2|2} ) \longrightarrow H^0 (\mathcal{T}_{\proj {1|2}} ) \longrightarrow H^1(\mathcal{O}_{\proj {1|2}}) \longrightarrow 0
\end{align}
Distinguishing between even and odd dimensions we finds the following results: 
\bear
\left \{ \begin{array}{l}
h^0 (\mathcal{O}_{\proj {1|2}}) = 1|0 \\
h^1 (\mathcal{O}_{\proj {1|2}}) = 1|0 \\
h^0 (\mathcal{O}_{\proj {1|2}}(+1)^{\oplus 2|2}) = 8 | 8, 
\end{array}
\right.
\eear
therefore we have that we can find
\begin{align} 
0\to \mathbb{C}^{1|0} {\longrightarrow} \mathbb{C}^{8|8} \longrightarrow H^0 (\mathcal{T}_{\proj {1|2}} ) \longrightarrow \mathbb{C}^{1|0} \to 0
\end{align}
so as for the dimensions we have 
\bear
h^0 (\mathcal{T}_{\proj {1|2}}) =  8|8 + 1|0 - 1|0 = 8|8. 
\eear
This is somehow surprising for this dimension does not correspond to the dimension of the super Lie algebra $\mathfrak{sl} (2|2)$, connected to $PGL (2|2)$: we would indeed find $\dim \mathfrak{sl} (2|2) = 7|8 \neq 8 |8$! \\
The SCY variety $\proj {1|2}$ stands out as the \emph{unique} exception among projective super spaces having $h^{0} (\mathcal{T}_{\proj {n|m}}) \neq \dim \mathfrak{sl} (n|m)$ (see \cite{1DCY}). There is indeed one more \virgolette infinitesimal automorphism'' to account for this correction to $\mathfrak{sl} (2|2)$, this is given by the everywhere defined field $\theta_1 \theta_2 \partial_z  \in H^0 (\mathcal{T}_{\proj {1|2}})$ (here represented in one of the two chart covering $\proj {1|2}$): this is the \emph{only} existing \emph{bosonisation} of the even (local) coordinate $z$.\\
Actually, one might think that following the same line - that is considering bosonisations of the even coordinates - one might discover many more everywhere-defined vector fields enlarging the symmetry transformation of $\proj {n|m}$. It is not like this, as the above computation in cohomology makes manifest. Indeed, such supposedly everywhere defined vector fields are not allowed by the transformation properties of the coordinates $\proj {n|m}$: the correct compensations that makes them into global vector fields happen only in the case on one even and two odd coordinates, corresponding to $\proj {1|2}$. The reader might convince himself by considering the $\theta\theta$-bosonisation in the case of $\proj {1|3}$ or $\proj {2|2}$. Going up in the order of bosonisation only makes the situation worse

We can indeed be more explicit by finding a basis of global sections. The most generic section, (in the local chart of coordinates $(z | \theta_1,\theta_2)$, has the form
\begin{align}
 s(z,\theta_1,\theta_2) =& (a(z)+b_1(z)\theta_1+b_2(z)\theta_2+c(z)\theta_1\theta_2)\partial_z\cr
 &+\sum_{i=1}^2(A^{(i)}(z)+B_1^{(i)}(z)\theta_1+B_2^{(i)}(z)\theta_2+C^{(i)}(z)\theta_1\theta_2)\partial_{\theta_i}.
\end{align}
By passing to the chart $(w | \phi_1,\phi_2)$ one has the transformation
\begin{align}
 z=\frac 1w, \quad\ \theta_i=\frac {\phi_i}w, \qquad i=1,2,
\end{align}
so that the local generators $\{ \partial_z , \partial_{\theta_i}\}$ for $i = 1, 2$ of $\mathcal{T}_{\proj {1|2}}$, transform as
\begin{align}
\partial_z=-(w^2\partial_w+w\phi_1\partial_{\phi_1}+w\phi_2\partial_{\phi_2}), \quad\ \partial_{\theta_i}=w\partial_{\phi_i}, \qquad  i=1,2.
\end{align}
Imposing the absence of singularities when changing local charts, from $(z|\theta_1, \theta_2)$ to $(w|\phi_1 , \phi_2)$ - that is computing $H^0 (\mathcal{T}_{\proj {1|2}})$ - we get the following 
\begin{teo}[Global Sections of $\mathcal{T}_{\proj {1|2}}$] \label{global} A basis of the vector superspace $H^0(\mathcal{T}_{\proj {1|2}})$ is given by the sections
\begin{align}
 & V_1=\partial_z, \qquad\ V_2=z\partial_z, \qquad\ V_3=z^2\partial_z+z\theta_1\partial_{\theta_1}+z\theta_2\partial_{\theta_2}, \qquad\ V_4=\theta_1\theta_2\partial_z, \cr
 & V_5=\theta_1\partial_{\theta_1}, \qquad\ V_6=\theta_2\partial_{\theta_1}, \qquad\ V_7=\theta_1\partial_{\theta_2}, \qquad\ V_8=\theta_2\partial_{\theta_2}, \\
 & \Xi_1=\theta_1 \partial_z, \qquad\ \Xi_2=z\theta_1\partial_z+\theta_1\theta_2\partial_{\theta_2}, \qquad\ \Xi_3=\theta_2 \partial_z, \qquad\ \Xi_4=z\theta_2\partial_z-\theta_1\theta_2\partial_{\theta_1}, \cr
 & \Xi_5=\partial_{\theta_1}, \qquad\ \Xi_6=z\partial_{\theta_1}, \qquad\ \Xi_7=\partial_{\theta_2}, \qquad\ \Xi_8=z\partial_{\theta_2}.
\end{align}
\end{teo}
\noindent Notice that $h^0 (\mathcal{T}_{\proj 1|2}) = 8|8$, as expected upon using homological methods. In the next section we will start from these sections to study the $\mathcal{N}=2$ super Riemann surfaces structure we can endow $\proj {1|2}$ with.
\subsection{$\proj {1|2}$ as $\mathcal{N} = 2$ Super Riemann Surface}
We now make explicit the $\mathcal{N}=2$ super Riemann surface structure of $\proj {1|2}$. Following \cite{ManinNC}, a supermanifold $\mani$ of dimension $1|2$ can be endowed with a $\mathcal{N} = 2 $ super Riemann surface structure if the super tangent sheaf $\mathcal{T}_{\mani}$ has two $0|1$-dimensional sub-sheaves $\mathcal D_1$ and $\mathcal D_2$, locally generated by vector fields $D_1$, $D_2$ such that they are integrable, i.e. $D^2_i = fD_i$ for some odd function $f$, and $\mathcal  D_1$,  $\mathcal D_2$ and $\mathcal D_1\otimes \mathcal D_2$ generate $\mathcal{T}_\mani$ at \emph{any} point. We will call $\mathcal{D}_1$ and $\mathcal{D}_2$ the \emph{structure distributions} of the $\mathcal{N}=2$ super Riemann surface. \\
In other words, a $\mathcal{N} =2$ super Riemann surfaces can be described by a triple $(\mani, \mathcal{D}_1, \mathcal{D}_2)$ where $\mani$ is an ordinary complex supermanifold of dimension $1|2$ and $\mathcal{D}_1$ and $\mathcal{D}_2$ are the two structure distributions having the properties mentioned above. 
\begin{ttes}[$\mathbb{C}^{1|2}$ as $\mathcal{N} = 2$ super Riemann surface] In order to endow the complex superspace $\mani = \mathbb C^{1|2}$ with a $\mathcal{N} = 2$ super Riemann surface structure one takes the sub-bundles generated, for example, by the global sections
\begin{align}
 D_{0,1}=\partial_{\theta_1}+\theta_2 \partial_z, \qquad  D_{0,2}=\partial_{\theta_2}+\theta_1 \partial_z,
\end{align}
which are integrable (indeed $\{ D_{0,i} , D_{0,i}\} = 0 $ for $i= 1, 2$ ) and have anticommutator given by
\begin{align}
 \{ D_{0,1},D_{0,2} \}=2\partial_z,
\end{align}
so that $D_{0,1}, D_{0,2}, \{D_{0,1},D_{0,2}\}$ generate the whole $\mathcal{T}_{\mathbb{C}^{1|2}}$, that is 
\bear
\mathcal{T}_{\mathbb{C}^{1|2}} \cong \mbox{\emph{Span}}_\mathbb{C} \big \{ D_{0,1}, \, D_{0,2}, \, \{D_{0,1} , D_{0,2} \} \big \}.
\eear 
This is an example of non-compact $\mathcal{N}=2$ super Riemann surface.
\end{ttes}
\noindent Now a remark is in order: the \virgolette defining sections'' $\{D_{0,1}, \, D_{0,2}, \{D_{0,1}, D_{0,2} \}  \}$ for the $\mathcal{N}=2$ super Riemann surface structure of $\mathbb{C}^{1|2}$, remain global sections of tangent sheaf even for the supermanifold $\mani=\proj {1|2}$, since, looking at the previous theorem \eqref{global} one has  
\begin{align}
D_{0,1}=\Xi_3+\Xi_5, \qquad\ D_{0,2}=\Xi_1+\Xi_7, \qquad\ \partial_z=V_1.
\end{align}
The big difference resides in that \emph{such sections are not sufficient to generate the whole $\mathcal{T}_{\proj {1|2}}$}, since $\partial_z$ has a double zero in $w=0$! Indeed, since the Euler characteristic of $\proj {1}$ is $2$, any even section has indeed two zeros and the invertible sheaves we are looking for cannot be everywhere generated by a single section. We thus need to look for more general odd integrable sections. \\
The most general form that a odd global section can take is
\begin{align}
 D_{odd}=\sum_{i=1}^8 \alpha^i \Xi_i,
\end{align}
where $\Xi$'s that appeared in \eqref{global} are such that $\mbox{Span}_{\mathbb{C}} \{\Xi_1, \ldots, \Xi_8 \} \cong \left (H^0 (\mathcal{T}_{\proj {1|3}}) \right )_1$ and where $\alpha^i$ for $i = 1, \ldots, 8$ are complex constants. After imposing the integrability condition in the form $D_{odd}^2 = 0$, we get that one must have the following conditions satisfied
\begin{align}
\left \{ 
\begin{array}{l} 
\alpha_1\alpha_5+\alpha_7\alpha_3 =0,\\
\alpha_2\alpha_6+\alpha_8\alpha_4 =0,\\
\alpha_1\alpha_6+\alpha_2\alpha_5 +\alpha_3\alpha_8+\alpha_4\alpha_7 =0. 
\end{array}
\right.
\end{align}
Solving, we find the sections
\begin{align} \label{lineb1}
 D_1=&\alpha_1 (\Xi_3+\Xi_5)+\alpha_2(\Xi_4+\Xi_6), \\
 D_2=&\beta_1 (\Xi_1+\Xi_7)+\beta_2(\Xi_2+\Xi_8),  \label{lineb2}
\end{align}
again for $\alpha_1, \alpha_2, \beta_1, \beta_2 \in \mathbb{C}$. The anticommutator reads
\begin{align} \nonumber 
 \{D_1,D_2\}=2[\alpha_1\beta_1\partial_z+\alpha_2\beta_1(\theta_1\partial_{\theta_1}+z\partial_z)+\alpha_1\beta_2(\theta_2\partial_{\theta_2}+z\partial_z)+\beta_1\beta_2(z^2\partial_z+z\theta_1\partial_{\theta_1}
 +z\theta_2\partial_{\theta_2})].
\end{align}
Notice that the sections $D_1$ and $D_2$ of equations \eqref{lineb1}, \eqref{lineb2} can be re-written in the more meaningful form
\begin{align}
D_1=&[\alpha_1+\alpha_2(z-\theta_1\theta_2)](\Xi_3+\Xi_5),\\
D_2=&[\beta_1+\beta_2(z+\theta_1\theta_2)](\Xi_1+\Xi_7).
\end{align}
These expressions make apparent that $D_1$ and $D_2$ generate two invertible sheaves of rank $0|1$ - we call them $\mathcal{D}_1$ and $\mathcal{D}_2$ respectively as above - by varying the coefficients (as the zeros are moved everywhere), that is we have 
\bear
\mathcal{D}_1 \cong \mbox{Span}_{\mathbb{C}} D_1, \qquad \mathcal{D}_2 \cong \mbox{Span}_{\mathbb{C}} D_2.
\eear 
Also, we see that now $\mathcal  D_1$,  $\mathcal D_2$ and $\mathcal D_1\otimes \mathcal D_2$ generate
the whole $\mathcal{T}_{\proj {1|2}}$, since the triple $\{D_1, D_2, \{D_1, D_2\} \}$ does.

\

We can also investigate the \emph{automorphisms} of the $\mathcal{N}=2$ super Riemann surface structure. Indeed, the automorphisms of $\proj {1|2}$ are generated by the set of all global sections of $\mathcal T_{\proj {1|2}}$ determined above.
We have to select the sub-algebra of global sections acting internally on the invertible sheaves $\mathcal D_1$, $\mathcal D_2$ and $\mathcal D_1\otimes \mathcal D_2$. By a direct inspection we see that the automorphisms of the $\mathcal{N}=2$ super structure are generated by a $4|4$-dimensional linear superspace with basis given by $\{ U_1, \ldots U_4, \Sigma_1 , \ldots, \Sigma_4\}$, where
\begin{align}
 & U_1 \defeq V_1, \qquad  U_2 \defeq V_2+V_5, \qquad U_3 \defeq V_3, \qquad U_4\defeq V_2+V_8, \nonumber \\
 & \Sigma_1\defeq\Xi_1+\Xi_7, \qquad \Sigma_2\defeq \Xi_2+\Xi_8, \qquad \Sigma_3\defeq\Xi_3+\Xi_5, \qquad \Sigma_4\defeq\Xi_4+\Xi_6. 
\end{align}
These generators satisfy the super commutation relations
\begin{align*}
 & [U_1,U_2]=U_1, \quad [U_1,U_3]=U_2+U_4, \quad [U_1,U_4]=U_1, \\
 & [U_2,U_3]= U_3, \quad  [U_2,U_4]=0, \quad [U_3,U_4]=-U_3;\\ 
 \\
 & \{\Sigma_1,\Sigma_2\}=0, \quad  \{\Sigma_1,\Sigma_3\}=2 U_1, \quad \{\Sigma_1,\Sigma_4\}=2U_2, \\
 & \{\Sigma_2,\Sigma_3\}=2U_4, \quad \{\Sigma_2,\Sigma_4\}=2U_3, \quad \{\Sigma_3,\Sigma_4\}=0;\\
 \\
 & [U_1, \Sigma_1]=0, \quad\ [U_1,\Sigma_2]=\Sigma_1, \quad\ [U_1,\Sigma_3]=0, \quad\ [U_1,\Sigma_4]=\Sigma_3, \\
 & [U_2,\Sigma_1]=0, \quad\ [U_2,\Sigma_2]=\Sigma_2, \quad\ [U_2,\Sigma_3]=-\Sigma_3, \quad\ [U_2,\Sigma_4]=0, \\
 & [U_3,\Sigma_1]=-\Sigma_2, \quad\ [U_3,\Sigma_2]=0, \quad\ [U_3,\Sigma_3]=-\Sigma_4, \quad\ [U_3,\Sigma_4]=0, \\
 & [U_4,\Sigma_1]=-\Sigma_1, \quad\ [U_4,\Sigma_2]=0, \quad\ [U_4,\Sigma_3]=0, \quad\ [U_4,\Sigma_4]=\Sigma_4.
\end{align*}
Something better can be done in order to write the resulting superalgebra in a more meaningful and, in particular, physically relevant form. 
We define 
\begin{align}
& H \defeq U_1, \qquad \quad K \defeq  U_3, \qquad \quad D \defeq \frac{1}{2} \left (U_2 + U_4 \right ),  \qquad \quad Y \defeq \frac{1}{2}\left (U_2 - U_4 \right ), \nonumber \\
& Q_1 \defeq \frac{1}{\sqrt{2}}\left ( \Sigma_1 - i \Sigma_3 \right ), \quad Q_2 \defeq \frac{1}{\sqrt{2}} \left ( \Sigma_3 - i \Sigma_1 \right ), \quad S_1 = - \frac{1}{\sqrt{2}} \left ( \Sigma_2 - i \Sigma_ 4 \right ), \quad S_2 \defeq - \frac{1}{\sqrt{2}} \left ( \Sigma_4 - i \Sigma_2 \right ).
\end{align} 
For completeness, we write these elements in terms of the (local) basis of the tangent space:
\begin{align}
& \mbox{{\bf Bosonic generators:} \; \; \; \;} 
 \left \{ \begin{array}{l}
H \defeq \partial_{z},\\
K \defeq   z^2\partial_z+z\theta_1\partial_{\theta_1}, \\
D \defeq z \partial_z + \frac{1}{2}(\theta_1 \partial_{\theta_1} + \theta_2 \partial_{\theta_2}), \\
Y \defeq \frac{1}{2}\left ( \theta_1 \partial_{\theta_1} -  \theta_2 \partial_{\theta_2} \right ); 
 \end{array}
 \right.
 \nonumber \\
& \mbox{{\bf Fermionic generators:} \; \; }  
 \left \{ \begin{array}{l}
Q_1 \defeq \frac{1}{2}\left ( \theta_1 \partial_{z} + \partial_{\theta_2} - i (\partial_{\theta_1} + \theta_2 \partial_z )\right ), \\
Q_2 \defeq \frac{1}{2}\left ( \partial_{\theta_1} + \theta_2 \partial_z  - i (\theta_1 \partial_{z} + \partial_{\theta_2} ) \right ), \\ 
S_1 \defeq \frac{1}{2} \left ( (- z \theta_1 + i z \theta_2 ) \partial_z + i(z  - \theta_1 \theta_2) \partial_{\theta_1} + (- \theta_1 \theta_2 -  z) \partial_{\theta_2}  \right ), \\
S_2 \defeq \frac{1}{2} \left ( (- z\theta_2 + i z\theta_1)\partial_z + (- z + \theta_1 \theta_2 )\partial_{\theta_1} + i ( z + \theta_1\theta_2 ) \partial_{\theta_2}\right ). 
\end{array}
\right.
\end{align} 
These definitions allows to prove, by simply computing the supercommutators, the following
\begin{teo}[$\mathcal{N} = 2$ SUSY Algebra] Let $(\proj {1|2}, \mathcal{D}_1, \mathcal{D}_2 )$ be the $\mathcal{N} = 2$ super Riemann surfaces constructed from $\proj {1|2}$. Then the algebra of the $\mathcal{N}=2$ SUSY-preserving infinitesimal automorphisms is generated by $ \Big \{H, K, D, Y \, \big | \, Q_1, Q_2, S_1, S_2 \Big \}$ and it corresponds to the Lie superalgebra $\mathfrak{osp} (2|2)$ of the orthosymplectic Lie supergroup $OSp(2|2) $, as it satisfies the following structure equations: 
\begin{align}
& \{ Q_i, Q_j \} = - 2 i \delta_{ij} H, \qquad \{S_i, S_j \} = - 2 i \delta_{ij } K, \qquad \{ Q_i , S_j \} = + 2 i \delta_{ij} D -  2 \epsilon_{ij} Y, \nonumber \\
& [H, Q_i ] =  0, \qquad [H, S_i ] = - Q_i, \qquad [H, Q_i]  = S_i, \qquad [H, S_i] = 0, \nonumber \\
& [D, Q_i ] = - \frac{1}{2} Q_i, \qquad [D, S_i] =  \frac{1}{2} S_i, \qquad [Y, Q_i] =  \frac{1}{2}\epsilon_{ij}Q_j, \qquad [Y, S_i] =  \frac{1}{2}\epsilon_{ij} S_j, \nonumber \\
& [Y, H] =0 , \qquad  \quad [Y, D] = 0, \qquad  \quad [Y, K] = 0,
\end{align}
together with the structure equations of the closed (bosonic) sub-algebra $\mathfrak{o} (2,1)$:
\begin{align}
[H, D] = H, \qquad [H, K] = 2D, \qquad [D, K] = K.
\end{align}
\end{teo}
\noindent We stress that, as the reader with some expertise in supersymmetric QFT's might have easily noticed, the form above has the merit to make manifest all the physically relevant elements of the superalgebra, such as the translations, rotations, supersymmetries, dilatations and so on. This shows a direct connection with physical theories, which is sometimes left hidden in the more mathematical oriented literature.  \vspace{.5cm}

It is anyway fair to stress that some attention need to to be paid here. Indeed, even if $\mathfrak{osp} (2|2)$ is actually the Lie superalgebra of automorphisms of $\proj {1|2}$ as $\mathcal{N}=2$ super Riemann surface, the related supergroup $OSp(2|2)$, defined as
\bear
OSp(2|2) \defeq \left \{ A \in GL(2|2): \; A^{st}I_{2|2}A = I_{2|2} \right \} \quad \mbox{where}\quad I_{2|2} \defeq \left ( \begin{array}{cc|cc} 
0 & - 1 & 0 & 0 \\
-1 & 0 & 0 & 0 \\
\hline 
0 & 0 & 0 & 1 \\
0 & 0 & -1 & 0  
\end{array}
\right ),
\eear 
is \emph{not} the supergroup of automorphisms of $\proj {1|2}$ as $\mathcal{N}=2$ super Riemann surface. Instead, it turns out (see for example \cite{ManinNC}) that the supergroup of automorphisms of $\proj {1|2}$ as a $\mathcal{N} =2 $ super Riemann surfaces - call it $\proj {1|2}_{\mathcal{N }=2}$ is obtained as a suitable quotient of $OSp(2|2)$, indeed we have
\bear
\xymatrix{ 
1 \ar[rr] && \mathbb{Z}_2 \ar[rr] && OSp (2|2) \ar[rr] && Aut\, (\proj{1|2}_{\mathcal{N}=2}) \ar[rr] && 1.
} 
\eear 
where $\mathbb{Z}_2 = \{ \pm 1 \}$. One can see that $Aut (\proj{1|2}_{\mathcal{N}=2})$ has two connected components as $OSp (2|2)$: an automorphism that do not belong to the identity component interchanges the two structure distributions, $\mathcal{D}_1 \leftrightarrow \mathcal{D}_2$.  
 
\subsection{$\mathcal{N}=2$ Semi-Rigid Super Riemann Surfaces and Genus $0$ Topological String} 
 
In the previous section we have shown how to endow $\proj{1|2}$ with the structure of an $\mathcal{N}=2$ super Riemann surface and we have studied its geometry to some extent. Actually, this particular supermanifold plays a fundamental role in a certain formulation of \emph{topological string theory}.

Roughly speaking, a topological string theory is obtained by coupling a topological field theory - that is a quantum field theory whose correlation functions can be exactly solved -, with a suitable worldsheet gravity (called \emph{$D=2$ topological gravity}). It has been observed long ago in \cite{DistlerNelson} that the worldsheet of a topological string theory can be constructed geometrically in a very natural way as a \emph{$\mathcal{N}= 2$ semi-rigid super Riemann surface}: these supermanifolds are $\mathcal{N}=2$ super Riemann surface whose fermionic sheaf $\mathcal{F}_{\mani_{\mathcal{N}=2}} = \mathcal{D}_1 \oplus \mathcal{D}_2$ undergoes a certain \emph{twist}.\\
Even if there exists some recent physical literature on the subject (see for example \cite{BeiJia}), only the big-picture is addressed and no specific example is actually carried out in detail. We leave to a forthcoming paper the actual construction of genus 0 topological string worldsheet and the related \emph{$A$-model} and \emph{$B$-model}. Here we limit ourselves to construct the only genus zero $\mathcal{N}=2$ semi-rigid super Riemann surface (indeed the related supermoduli stack has dimension $0|0$ if $g=0$, it is a superpoint). The general recipe to obtain a semi-rigid super Riemann surface from an $\mathcal{N}=2$ super Riemann surface is to twist its distinguished sheaves, call it $\mathcal{D}_+$ and $\mathcal{D}_-$ as to get an odd trivial sheaf (together with an odd non-zero trivial section) and an odd canonical sheaf. In the case of $\proj{1|2}_{\mathcal{N}=2}$ the twist leads then to $\Pi \mathcal{O}_{\proj 1}$ and $\Pi \mathcal{O}_{\proj 1} (-2),$ and there are two such possible twists:  
\begin{align}
\mbox{{\bf Topological Twists:}} \qquad \left \{ \begin{array}{l}  \mathcal{TW}_+ : \mathcal{D}_+ \oplus \mathcal{D}_- \longrightarrow \Pi \mathcal{O}_{\proj 1} \oplus \Pi \mathcal{O}_{\proj 1} (-2),\\ \\
 \mathcal{TW}_- : \mathcal{D}_+ \oplus \mathcal{D}_- \longrightarrow \Pi \mathcal{O}_{\proj 1}(-2) \oplus \Pi \mathcal{O}_{\proj 1},
\end{array}
\right.
\end{align}
We call the related semi-rigid super Riemann surfaces $\proj {1|2}_{\mathcal{N}=2, +}$ and $\proj {1|2}_{\mathcal{N}=2, -}$: clearly one has $\proj {1|2}_{\mathcal{N}=2, +} \cong \proj {1|2}_{\mathcal{N}=2, -}.$ These two supermanifolds will be the basic ingredients that enter the construction of genus 0 topological string worldsheet and different choices of twists for certain pairs of $\mathcal{N}=2$ semi-rigid super Riemann surface yields $A$-model or $B$-model, but we will need to deal with some subtleties concerning embedding of supermanifolds.

\subsection{The Calabi-Yau Supermanifold \proj{3|4}}

We now briefly discuss the Calabi-Yau supermanifold $\proj {3|4}$. This is definitely one of the most celebrated supermanifold in theoretical physics, where it is sometimes called \emph{supertwistor space}. The relevance of $\proj {3|4}$ is mainly due to one \cite{Witten}. In this famous paper, Edward Witten showed that there exists an \virgolette equivalence'' between the perturbative expansion of $\mathcal{N} = 4$ super Yang-Mills theory and the $D$-instantons expansion of a particular string theory. This string theory is actually a \emph{topological B model} having as target space the Calabi-Yau supermanifold $\mathbb{P}^{3|4}.$  \\ 
Even if $\proj{3|4}$ enters this striking and unexpected duality, it is fair to say that on the supergeometrical ground it is a pretty rough supermanifold, as we shall see shorly.\\
As discussed early on in this paper, its automorphisms supergroup is given - as expected - by $PGL (4|4)$, of dimension $31|32$. This matches the calculation at the level of the algebra of the automorphisms that can be done using the Euler exact sequence, 
\bear
\xymatrix{ 
0 \ar[rr] && \mathcal{O}_{\proj {3|4}} \ar[rr] && \mathcal{O}_{\proj {3|4}}(+1)^{\oplus 4|4} \ar[rr] && \mathcal{T}_{\proj {3|4}} \ar[rr] && 0.
} 
\eear 
The induced long cohomology exact sequence begins with
\bear
\xymatrix{ 
0 \ar[r] & H^0(\mathcal{O}_{\proj {3|4}}) \ar[r] & H^0(\mathcal{O}_{\proj {3|4}}(+1))^{\oplus 4|4} \ar[r] & H^0(\mathcal{T}_{\proj {3|4}}) \ar[r] & 0,
} 
\eear 
then using the above formulas or with a direct calculation one easily finds that, on the one hand $H^0 (\mathcal{O}_{\proj {3|4}}) = \mathbb{C}^{1|0}$, while on the other hand
\begin{align}
H^0 (\mathcal{O}_{\proj {3|4}}(+1)^{\oplus 4|4}) \cong H^0 (\mathcal{O}_{\proj 3} (+1)^{\oplus 4} \oplus \mathcal{O}_{\proj 3}^{\oplus 16}) \oplus \Pi \,H^0 (\mathcal{O}_{\proj 3} (+1)^{\oplus 4} \oplus \mathcal{O}_{\proj 3}^{\oplus 16}) \cong \mathbb{C}^{32|32},
\end{align}
so that taking the quotient one has $H^0 (\mathcal{T}_{\proj {3|4}}) = \mathbb{C}^{31|32}$.

Moreover, going up in the long cohomology exact sequence and looking again at the sheaves on $\proj {3|4}$ as sheaves on $\proj {3}$ by means of the projection $\pi : \proj {3|4} \rightarrow \proj 3$, one sees that $H^1 (\mathcal{T}_{\proj {3|4}}) = 0$. This is so since one has $H^i (\mathcal{O}_{\proj 3} (k) ) = 0$ for $i=1, 2$ for any $k\in \mathbb{Z}$ and therefore $H^1 (\mathcal{T}_{\proj {2|3}})$ sits between two zeroes. This tells that the Calabi-Yau supermanifolds $\proj {3|4}$ is actually \emph{rigid}, it does not allow (first order) deformations (see \cite{Vaintrob}). Such a geometric rigidity result, in turn, calls for a better understanding of the \emph{super mirror map} proposed by Aganagic and Vafa in \cite{AgaVafa}: this involves in particular $\proj{3|4}$, whose supposed mirror supermanifold would be (in a suitable limit) a certain quadric in $\proj {3|3} \times \proj {3|3}$. The whole construction is carried out by mean of a path-integral computation and it is hard to retrive the underlying geometry and the actual mirror map. For example, if in the ordinary geometric formulation of mirror symmetry the K\"ahler moduli and the complex moduli of two mirror manifolds get interchanged, in a supergeometric setting it is not even clear how to define a K\"ahler structure - and thus a K\"alher moduli space -, so that it is hard to say in what sense the map given in \cite{AgaVafa} should be intended as a mirror map, rather than something else.

\appendix

\section{Even Exponential Map}


\noindent In this appendix we prove the actual exactness of \emph{even} exponential exact sequence, that we use throughout this paper to study the invertible sheaves of rank $1|0$ on a supermanifold,  
\bear
\xymatrix{
0 \ar[rr] && \mathbb{Z}_{\mani}  \ar[rr] &&  \mathcal{O}_{\mani, 0}  \ar[rr]^{\exp} &&  \mathcal{O}_{\mani, 0}^\ast \ar[rr] & & 0,
}  
\eear
where $\mathcal{O}_{\mani, 0} $ is the even part of the structure sheaf of $\mani$ and $\mathcal{O}_{\mani, 0}^\ast$ is the even part of the sheaf of invertible elements in $\mathcal{O}_{\mani}$, i.e. whose sections have a non-zero reduced part. Notice that they are sheaves of \emph{abelian} groups (respectively, additive and multiplicative).\\
Actually, the only thing that we need to check is the surjectivity of the exponential map. Given a supermanifold $\mani$, we therefore consider
\begin{align}
\mathcal{U} \longmapsto \exp_{\mathcal{U}} \; : \mathcal{O}_{\mani,0}(\mathcal{U}) & \longrightarrow \mathcal{O}_{\mani, 0}^{\ast}(\mathcal{U}) \nonumber \\
s_{0} & \longmapsto \exp_{\mathcal{U}}(s_0) \defeq e^{2\pi i s_0}.
\end{align}
where $\mathcal{O}_{\mani,0}$ and $\mathcal{O}^{\ast}_{\mani,0}$ are the \emph{even} part of the corresponding $\mathbb{Z}_2$-graded sheaf. We have the following
\begin{teo} The map $\exp$ defined above is surjective and $\ker (\exp) = \mathbb{Z}_\mani$. 
\end{teo}
\begin{proof} Surjectivity is to be proved locally, on the stalks. Choosing an open set $\mathcal{U} \owns x$, we can take a representative of an element in $\mathcal{O}^\ast_{\mani,0,\, x}$ such that the corresponding element in $\mathcal{O}^{*}_{\mani, 0}(U)$ has the following expansion
\bear
f_0 (x, \theta ) =  f_{\emptyset} (x) + N (x ,\theta), \qquad f_{\emptyset } (x) \neq 0.
\eear
Notice that, for the sake of convenience, we have split the contribution on the reduced manifold, $f_{\emptyset} (z) $ - which is an ordinary non-zero holomorphic function since we are considering an invertible element in $\mathcal{O}^{\ast}_{\mani,0}(\mathcal{U})$ - and we have gathered all the nihilpotent contributions in the expansion in the term $N(x,\theta) \in \mathcal{J}_{\mani}(\mathcal{U})$, such that $N^m (x, \theta) = 0$ and $N^{m-1} (x , \theta) \neq 0$ for some $m \geq 2,$ nihilpotency index. \\
Now, since $f_{\emptyset} (x) \neq 0$, if one wish, it can be collected to give
\bear
f_0 (x, \theta ) = f_{\emptyset}(x) \left ( 1+ \frac{N(x, \theta)}{f_{\emptyset} (x)} \right ).
\eear
This might be useful in writing the logarithm, defined as to be the (local) inverse of the exponential, that is $\mathcal{U} \mapsto \log_{\mathcal{U}} $ with $\log_\mathcal{U} (s_0) = \frac{1}{2\pi i} \log (s_0)$ for $s_0 \in \mathcal{O}_{\mani,0}^{\ast}$. In this way, using the expression above, one finds:  
\begin{align}
\log_{\mathcal{U}} ( f_0 ) & = \frac{1}{2\pi i}\log \left ( f_{\emptyset}(x) \right ) + \frac{1}{2\pi i}\log \left (1 +  \frac{N(x, \theta)}{f_{\emptyset} (x)} \right ) \\
& = \frac{1}{2\pi i}\log  \left ( f_{\emptyset}(x) \right ) + \frac{1}{2\pi i}\sum_{k=0}^{m-2} \frac{(-1)^k}{k+1} \left ( \frac{N(x,\theta)}{f_{\emptyset} (x)} \right )^{k+1} 
\end{align}
This is well-defined for $\log  \left ( f_{\emptyset}(x) \right )$ is the logarithm of an ordinary holomorphic non-zero function and it is locally single-valued and the remaining part is a finite sum of nihilpotents. Therefore one has that over a generic open set $\mathcal{U} \subset |\mani|$ containing $x$, $f_0 = \exp_\mathcal{U} (\log_\mathcal{U} (f_0)$, that is $\exp$ is surjective.
We can now evaluate the exponential of the above quantity to establish the kernel of the map:  
\begin{align}
\exp_\mathcal{U} (f_0) & = e^{2 \pi i \left ( f_{\emptyset} (x) + N (x ,\theta)\right )} = e^{2 \pi i f_{\emptyset} (x)} e^{2 \pi i N(x , \theta)} = \\
& = 2\pi i \, e^{2 \pi i f_{\emptyset} (x)} \left ( 1 + \sum_{k =1}^{m-1} \frac{N(x, \theta)^k}{k!} \right ) \stackrel{!}{=} 1_\mathcal{U}
\end{align}
Now the exponential above, $e^{2\pi i f_{\emptyset } (x)}$, is the usual complex exponential map that has kernel given by the sheaf of locally constant functions taking integral values $\mathbb Z$. Let suppose that $f_0 \in \ker (\exp) $, the only way for this to be true is that $\sum_{k =1}^{m-1} \frac{N(x, \theta)^k}{k!}= 0$, which in turn implies that $N(x, \theta) = 0$, indeed, multiplying on the left and on the right side by $N^{m-2}$: 
\bear
\left ( \sum_{k =1}^{m-1} \frac{N(x, \theta)^k}{k!} \right )\cdot N^{m-2} (x , \theta) = N^{m-1} (x, \theta) \neq 0. 
\eear 
 \end{proof}

\end{document}